\documentclass[11pt,a4paper]{article}
\usepackage{amssymb, amstext, amscd, amsmath, color}
\usepackage{graphicx, url, hhline}
\usepackage{mathrsfs} 
\usepackage[left=2.50cm, right=2.00cm, top=2.30cm, bottom=2.00cm]{geometry}

\usepackage{float}

\usepackage{longtable}
\usepackage{multicol}

\newtheorem{defin}{}

\newtheorem{conjec}[defin]{}
\newtheorem{lemmas}[defin]{}

\newtheorem{bemerk}[defin]{}
\newtheorem{prop}{Proposition} 
\newtheorem{theorem}{Theorem}
\newtheorem{corollary}{Corollary}

\newenvironment{proof}    {{\noindent \it Proof}:}{{\fillbox \medskip}} 

\newcommand{\fillbox}{\mbox{$\bullet$}}

\newcommand{\ti}{\tilde}
\newcommand{\N}{\mathbb N}

\newcommand{\Z}{\mathbb Z}

\newcommand{\R}{\mathbb R}
\newcommand{\GAP}{\mathsf{GAP}}

\newcommand{\bolit}{\boldsymbol}

\newcommand{\Aff}{\operatorname{Aff}}

\begin{document}

\bigskip
\bigskip

\centerline{\large \textbf{Subperiodic groups and bounded automorphisms of periodic graphs}}

\bigskip
\bigskip

\centerline{\large Igor A. Baburin}
\centerline{\emph{Email: baburinssu@gmail.com}}

\bigskip

\noindent A subperiodic group is a group of motions of $d$-dimensional Euclidean space $\R^d$ which contains a translation lattice $\Z^r$ of rank $r < d$ as a subgroup of finite index. Cayley graphs of subperiodic groups are used to illustrate the concept of bounded automorphisms (particularly of \emph{finite order}), \emph{i.e.} those when the distance between a graph vertex and its image has an upper bound. The description of subperiodic groups in dimension~3 in terms of subgroup structure and abstract group isomorphism types [75 \emph{crystallographic} rod groups ($r=1$) and 80 layer groups ($r=2$) belong to 32 and 34 types, respectively] allows for a convenient characterization of bounded automorphisms of their Cayley graphs. An easy-to-compute set of invariants is developed for recognizing the isomorphism types of crystallographic groups in~$\R^3$ from finite presentations which makes use of only the number of nonconjugate subgroups up to a given finite index~$n$ ($n \leq 12$ for rod groups, $n \leq 8$ for layer groups and $n \leq 24$ for space groups), and how many of them are normal. A~criterion is proven that if a finite (inverse-closed) generating set of a group $G$ is stabilized by conjugation by a nontrivial element of finite order in~$G$, then the respective Cayley graph possesses nontrivial bounded automorphisms of finite order. This criterion is applied to three-periodic sphere-packing graphs. Subperiodic groups in~$\R^4$ with a three-dimensional translation lattice are used to systematically derive embeddings of three-periodic \emph{ladder graphs} in~$\R^3$.

\section{Introduction}

Crystallographic space groups and subperiodic groups are the groups of isometries of $d$-dimensional Euclidean space $\R^d$ which are distinguished by the rank of the maximal translation subgroup $\Z^r$: $r=d$ holds for the former, and $r < d$ for the latter. Let $G$ be a space group or a subperiodic group. In algebraic terms $G$ is an extension of a free abelian group $T \cong \Z^r$ of rank $r$ by a finite group $P = G/T$. The group $P$ is the point group of $G$, and $T$ is its translation subgroup (translation lattice) that we shall also denote by $P(G)$ and $T(G)$, respectively. In the case of space groups ($r=d$), $P$ acts faithfully on $T$ that is generally not true for subperiodic groups \cite{Koe80}. Let $H$ be a subperiodic group in $\R^d$ with translation lattice~$\Z^r$. Then

\begin{enumerate}
\setlength\itemsep{-0.02in}
\item[(a)] $H$ can be isomorphic to some space group in $\R^r$,
\item[(b)] $H$ can have non-trivial finite normal subgroups [owing to the non-faithful action of $P(H)$ on $T(H)$], in which case it is not isomorphic to a space group of $\R^d$ for any $d$,
\item[(c)] subperiodic groups can be isomorphic without being affinely equivalent (\emph{e.g.} groups with non-isomorphic point groups can be isomorphic),
\end{enumerate}

\noindent where affine equivalence of groups means, as usual, their conjugacy in the group of all affine mappings of~$\R^d$.

\medskip

\noindent \emph{Example}: Consider the following extensions of a 1-dimensional lattice by finite groups generated by mirror reflections in coordinate planes and 2-fold rotations about coordinate axes: $H_1 = p2mm = \langle m_{xy}, m_{xz}, t_z \rangle$, $H_2 = pmm2 = \langle m_{xz}, m_{yz}, t_z \rangle$ and $H_3 = p222 = \langle 2_{x}, 2_{y}, t_z \rangle$, where $t_z$  stands for the unit translation along the $z$-axis. The non-isomorphism of $H_1$ and $H_2$ can be seen if their maximal finite normal subgroups are regarded: such a subgroup of $H_1$ is generated by one reflection whereas that of $H_2$ is generated by two reflections. As a result, the following factorisations are obtained: $H_1 \cong H_3 \cong D_{\infty} \times C_{2}$ and $H_2 \cong C_{\infty} \times D_{2}$. Being isomorphic as abstract groups, $H_1$ and~$H_3$ are not affinely equivalent since an affine transformation does not map (by conjugation) mirror reflections to rotations, and \emph{vice versa}.

\medskip

Subperiodic groups in dimension 3 with translation lattices of rank 1 and 2 are catalogued in Volume~E of the \emph{International tables for crystallography} (2006) where conventional coordinate description is given for 75~\emph{crystallographic} rod groups ($r=1$) and 80~layer groups ($r=2$). The numbers of groups originate from the conjugacy in the group of all \emph{orientation-preserving} affine mappings of $\R^3$ -- \emph {i.e.} from the same equivalence relation that yields 230 space groups in $\R^3$. Note that the finite number of (crystallographic) rod groups is only due to the restriction on orders of rotation axes (1, 2, 3, 4, 6) that is arbitrary for $r=1$.

Subperiodic groups can also be viewed as quotient groups of space groups in $\R^d$ admitting (maximal) normal translation subgroups of rank $r < d$ \cite{Litv2000}. For example, every hexagonal space group~$G$ in $\R^3$ has normal rank-1 and rank-2 translation subgroups generated by unit translations: $\langle t_z \rangle$ and $\langle t_x, t_y \rangle$. The respective quotients naturally give rise to the rod group $G/\langle t_x, t_y \rangle$ and layer group $G/\langle t_z \rangle$.

Subperiodic groups of $\R^3$ find applications in solid-state physics as symmetry groups of nanostructures (nanotubes and two-dimensional materials) where representation theory is the main tool \cite{Linegr10}. In this paper we shall develop a more abstract view of subperiodic groups motivated by application to automorphisms of vertex-transitive periodic graphs, specifically \emph{bounded} automorphisms of finite order. The history of the problem goes back to the beginning of 1980s when M.~Gromov proved his famous theorem on groups of polynomial growth \cite{Gromov81} (see Appendix~A for more details). Shortly after that, by extending this result, V.~I.~Trofimov characterized vertex-transitive graphs of polynomial growth (of which periodic graphs are a special case) in terms of the structure of automorphism groups, and in terms of partitions of vertex sets with (so to speak) convenient properties \cite{Trofimov84}. Solid-state structures dealt with in crystallography usually have translational symmetry that already implies polynomial growth (see below for the definition). Trofimov's theorem \cite{Trofimov84} thus tells us which additional automorphisms -- beyond the space-group symmetry -- periodic structures might possess. It is precisely where bounded automorphisms of finite order come into play. Eon (2005) rediscovered the idea of bounded automorphisms independently and introduced it to crystallography \cite{Eon2005}. Recently (\emph{ca.}~2019 onwards), the results of \cite{Trofimov84} have been proved (or reinterpreted) in the language of topological groups \cite{Salle19, Tes21} that allowed to obtain some new, including also more specialized, results.

In physics and chemistry it is often the case that certain mathematical results have to be \emph{felt} in addition to be rigorously \emph{proven}. In our opinion, bounded automorphisms of finite order in periodic graphs still remain somewhat obscure to crystallographers, and subperiodic groups provide a natural framework for a better understanding of this phenomenon. Broadly speaking, this paper is a further adaptation of \cite{Trofimov84} to periodic graphs, with the main emphasis on computational group-theoretic aspects. A more abstract treatment of subperiodic groups in relation to graph automorphisms also leads to interesting applications, for example, exploring non-isotopic embeddings of graphs in $\R^3$ which are used to complement a group-theoretic discussion. 

Let us outline the structure of the paper. Section~2 contains preliminaries on the terminology used. Section 3 deals with the classification of subperiodic groups in $\R^3$ into abstract group isomorphism classes. The language of generating sets and associated Cayley graphs is used throughout, also for illustrative purposes. Cayley graphs are always characterized with respect to bounded automorphisms of finite order. The logic of our exposition follows from that bounded automorphisms of finite order naturally become more and more rare on going from rod groups to space groups. Section~4 discusses in depth the automorphisms of Cayley graphs which are induced by different actions of a group. Finally, for Cayley graphs a group-theoretic criterion is proved that is sufficient for deciding whether they possess nontrivial bounded automorphisms of finite order. Section~4 includes various examples and describes some applications, mainly inspired from the analysis of sphere-packing graphs. 

\section{Preliminaries on groups and graphs} 

Our terminology for crystallographic groups is mostly standard (\emph{cf.} above). More general things shall be introduced here. For a group $G$, $|g|$ is the order of $g \in G$. The center of $G$ is denoted by $Z(G)$. For a subgroup $H<G$, $C_G(H)$ and $N_G(H)$ are its centralizer and normalizer in $G$, respectively. The index of $H$ in $G$ is~$|G:H|$. For a subset $S \subset G$, define its stabilizer in $G$ as $Stab_G(S) = \{ g \in G | g^{-1}Sg = S \}$.

A group is called an FC-group if conjugacy classes of all its elements are finite \cite{Tom84}. The~subgroup $H \leq G$ consisting of all the elements with finitely many conjugates is the FC-center of $G$. For example, the FC-center of a space group $G$ is its maximal translation subgroup $T(G)$.

All graphs in this paper are connected, simple and \emph{locally finite} (all vertices have finite valency). The vertex set of a graph $\Gamma$ is denoted by $V(\Gamma)$. The distance between vertices $x$ and $y$ in a (connected) graph, $d(x, y)$, is the number of edges in a shortest path (a geodesic) connecting them. The neighbourhood of vertex~$x$ is $\{y \in V(\Gamma) | d(x, y) = 1\}$. The automorphisms of $\Gamma$ are understood as permutations on $V(\Gamma)$ preserving both adjacency and non-adjacency of vertices. The group of all automorphisms of $\Gamma$ is denoted by $Aut(\Gamma)$. For $G \leq Aut(\Gamma)$, the stabilizer of vertex $x$ is $G_x = \{ g \in G | gx=x \}$. The setwise stabilizer of $X \subset V(\Gamma)$ is defined as $G_{\{X\}} = \{g \in G|gX = X\}$. If $\Gamma$ is a vertex-transitive graph, then a group $G \leq Aut(\Gamma)$ acting \emph{freely} and \emph{transitively} on $V(\Gamma)$ is called a \emph{regular} group of automorphisms.

Given a group $G$ and a finite inverse-closed generating set $S = S^{-1} \subseteq G \setminus \{1_G\}$, recall that the \emph{Cayley graph} of $G$ with respect to $S$ is the connected graph whose vertices are elements of $G$, and vertices $\{g, h\}$ are joined by an (unoriented) edge if (and only if) $g^{-1}h \in S$, or equivalently, $h=gs$ for some $s \in S$. The action of $G$ on itself by left multiplication induces a regular group of automorphisms of a Cayley graph \cite{Biggs93}.

An automorphism $g$ of $\Gamma$ is termed \emph{bounded} if there is a constant $C < \infty$ (depending on~$g$) so that $d(x, gx) < C$ holds for every $x \in V(\Gamma)$ \cite{Trofimov83}. For Cayley graphs of groups (when $V(\Gamma) = G$), this is equivalent to saying that the conjugacy class of $g$ is finite: since $d(x, gx) = d(1_G, x^{-1}gx) < C$, the latter holds, if and only if $\{x^{-1}gx | x \in G\}$ is finite \cite{Seif89all, Seif91}. Bounded automorphisms form a normal subgroup $B(\Gamma)$ of $Aut(\Gamma)$.

In the focus of our paper are (connected) $r$-periodic graphs which admit $\Z^r$ ($r \geq 1$) as a group of automorphisms acting freely with a finite number of both vertex- and edge-orbits \cite{Delgado2004}. Translations are therefore bounded automorphisms of infinite order. This implies that the balls about vertices in a graph, $B(x, n) = \{y \in V(\Gamma) |d(x, y) \leq n\}$, grow polynomially: $|B(x, n)| \leq c(x)n^r$ for all $n \in \Z_{\geq 0}$ and for any fixed $x \in V(\Gamma)$ where $c(x)$ is a constant that becomes independent of $x$ for vertex-transitive graphs \cite{Seif97}. The degree of polynomial growth is the exact lower bound for $r$ (equal to the rank of a free abelian group) which enters the above inequality. Note that polynomial growth is a far more general property than translational symmetry (\emph{cf.} Appendix~A). 

When discussing piecewise linear embeddings of periodic graphs in $\R^3$, we assume that vertices are mapped to points in $\R^3$, and  edges to straight-line segments. The edges intersect at most at common vertices. Two embeddings are called isotopic if (speaking informally) they can be continuously deformed one into another without breaking edges avoiding also edge intersections across the deformation path. In many cases the graphs we consider are contact graphs of sphere packings (whose vertices are spheres and edges correspond to sphere contacts) which we shall call \emph{sphere-packing graphs} as common in crystallography.

\section{Subperiodic groups in $\R^3$} 

\subsection{Rod groups} \label{rod1}

Table~\ref{t:table1} provides a classification of rod groups into 32 isomorphism classes along with some information on the subgroup structure. Each row corresponds to one isomorphism class, groups are designated following \cite{ITE}.  The relatively simple structure of rod groups often permits to prove group isomorphism by hand, especially if the groups to be compared can be factorized as direct products:

\begin{center}
\begin{tabular}{l}
$p\bar3m1 = p\bar1 \times C_{3v}$ \\ 
$p\bar6m2 = p11m \times C_{3v}$ \\
$p6_3/mmc = p112_1/m \times C_{3v}$,
\end{tabular}
\end{center}

\noindent where $p\bar1 \cong p11m \cong p112_1/m \cong D_{\infty}$, or if an isomorphic mapping of group generators can be easily seen. In many cases a convenient choice of generators can be made from the tabulation of subperiodic homogeneous sphere packings by Koch and Fischer (1978). For checking purposes group isomorphism was proven from group presentations \cite{Bab26} as illustrated by the two characteristic examples below.

The most unexpected group isomorphism is found for the pair ($p4_2/m$, $p\bar42c$). Given the generators 

\begin{center}
\begin{tabular}{l}
$a = \bar4_z(000) : y, -x, -z$,\\
$b = m(x, y, -1/4) : x, y, -1/2-z$,\\
$c = 2(x, 0, 1/4) : x, -y, 1/2-z$,
\end{tabular}
\end{center}

\noindent we have: $p4_2/m = \langle a, b \rangle$; $p\bar42c = \langle a, c \rangle$. Group presentations on these generators are identical: 

\begin{center}
\begin{tabular}{l}
$\langle a, b  | a^4, b^2, (aba)^2 \rangle \cong \langle a, c | a^4, c^2, (aca)^2 \rangle$.
\end{tabular}
\end{center}

\noindent We have checked that `analogous' isomorphism does not hold for groups containing principal axes of higher order [\emph{e.g.} ($p6_3/m, p\bar6c2$), ($p8_4/m, p\bar8c2$) \emph{etc.}]. 

\begin{center}
\begin{figure}[h]
\centering
\includegraphics{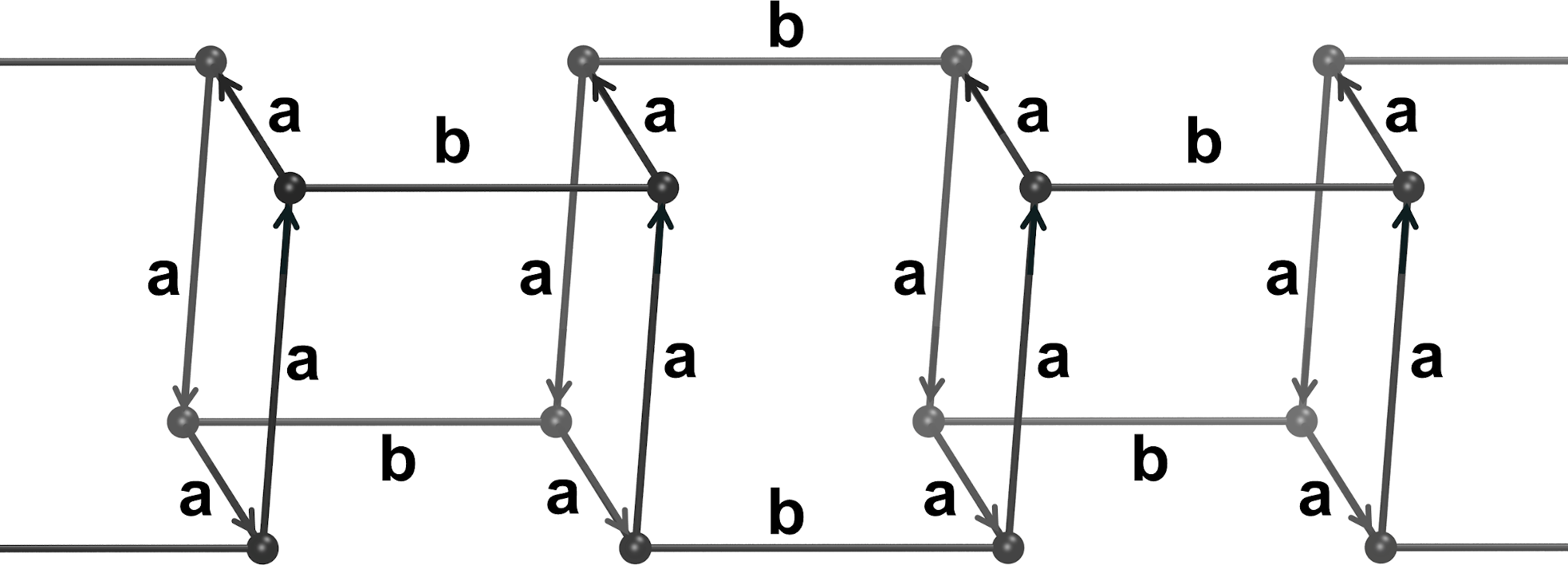}
\caption{A Cayley graph for the group defined by the presentation $\langle a, b|a^4, b^2, (aba)^2 \rangle$ (arrows distinguish between $a$~and~$a^{-1}$). See the text for the edge labels.}
\label{f:p42m}
\vspace{-2em}
\end{figure}
\end{center}

The pair ($p4_2/m$, $p\bar42c$) represents an example where distinct affine realisations of an abstract group act on the same embedding of the corresponding Cayley graph in $\R^3$ (Fig.~\ref{f:p42m}). The pairs of isomorphic groups ($p222, p4_222$), ($p321$, $p6_322$) correspond to a different situation that we shall discuss in a more general setting. Consider the following symmetry operations given below by the Seitz symbols [where $R_{\alpha}(\varphi)$ is a shorthand for the counterclockwise rotation by $\varphi = 2\pi/n$ ($n \in \N$) about the axis $\alpha$]:

\begin{center}
\begin{tabular}{l}
$a = \{R_z(\varphi)|000\}, b = \{R_x(\pi)|000\}, c = \{R_x(\pi)|001\}$, \\
$d = \{R_x(\pi)R_z(\varphi)|001\}, f = \{R_x(\pi)R_z(\varphi/2)|00\frac{1}{2}\}$.
\end{tabular}
\end{center}

\noindent Presentations for the two families of groups are as follows ($n \geq 2$)\footnote{For simplicity we write here the same symbol $pn21$, be $n$ even or odd.}: 

\begin{center}
\begin{tabular}{l}
$G_1 = pn21 = \langle a, b, c \rangle = \langle a, b, d \rangle$,\\
$G_2 = p(2n)_n22 = \langle a, b, f \rangle$, \\
$G_1 \cong G_2 \cong \langle a, b, s | a^n, b^2, s^2, (as)^2, (ab)^2 \rangle \, (s=c, d, f)$.
\end{tabular}
\end{center}

\noindent Cayley graphs for these generating sets are isomorphic to quotients of the square lattice with respect to translation subgroup $\langle (0, n) \rangle$ which we shall denote (following \cite{KF78}) by  $\bf 4^4(0, n)$. Although isomorphic, they show up remarkable differences in their embeddings in $\R^3$: with symmetry $\langle a, b, c \rangle$ `infinite lateral edges' of the graph are disjoint, whereas with symmetry $\langle a, b, d \rangle$ and $\langle a, b, f \rangle$ they are intervowen giving rise to the $n$-tuple helices (Fig.~\ref{f:3tubes}). As a consequence, the two embeddings are non-isotopic. This also manisfests itself in the type of a group-subgroup relation between the symmetry of the graph and the stabilizer of an `infinite lateral edge': the latter is a subgroup of index~$n$ in both $G_1$ and $G_2$. However, $H_1 = \langle b, c\rangle \cong p211$ is a \emph{translationengleiche} subgroup of~$G_1$ while $H_2 = \langle b, d\rangle \cong \langle b, f \rangle \cong pn_122$  is a \emph{klassengleiche} subgroup of $G_2$ with $n$ times translation period in the $z$-direction. Therefore, the nontrivial cosets of $T(H_2)$ in $T(G_2)$ map individual helices onto each other. A closer look at the embeddings of multiple helices with symmetry $p(2n)_n22$ has shown that this symmetry is optimal in the following sense: it yields vertex-transitive embeddings where edges correspond to shortest equal inter-vertex distances (`bonds'), whereas vertices of different helices could be moved sufficiently far apart. The relevance of such embeddings to crystal structure design has been discussed recently by O'Keeffe and Treacy (2021) \cite{OK21}.  

\medskip

\noindent \emph{Remark 1}. The same reasoning applies to generating symmetries of some 1-periodic sphere packings from Koch and Fischer (1978) [designated by $\bf 6^3(4, 2)$, $\bf 4^4(0, 2)$, $\bf 4^4(0, 3)$, $\bf 4^4(0, 4)$] which give rise to non-isotopic embeddings of sphere-packing graphs -- an aspect that remained unnoticed in their paper. For example, for the type $\bf 4^4(0,3)$ they tabulated the generating sets (in the above notation) $\{a, b, c\}$ and $\{a, b, f\}$ for rod groups $p321$ and $p6_322$, resp. Our numerical checks have shown that the set $\{a, b, d\}$ cannot produce a sphere packing of contact number~4. At first sight this discussion might seem as a pure exercise in geometry and group theory. Nevertheless, it is related to \emph{e.g.} mechanisms of crystal growth driven by the Eshelby twist \cite{growth21} and to conformations of nanotubes~\cite{Twist2012}. For nanotubes the energetically favourable configuration can be in general different from one obtained by the geometrical rolling procedure \cite{Evar09} that excludes structures like those of symmetry $p8_422$ and $p6_222$ in Fig.~\ref{f:3tubes}.

\begin{center}
\begin{figure}[ht]
\centering
\includegraphics{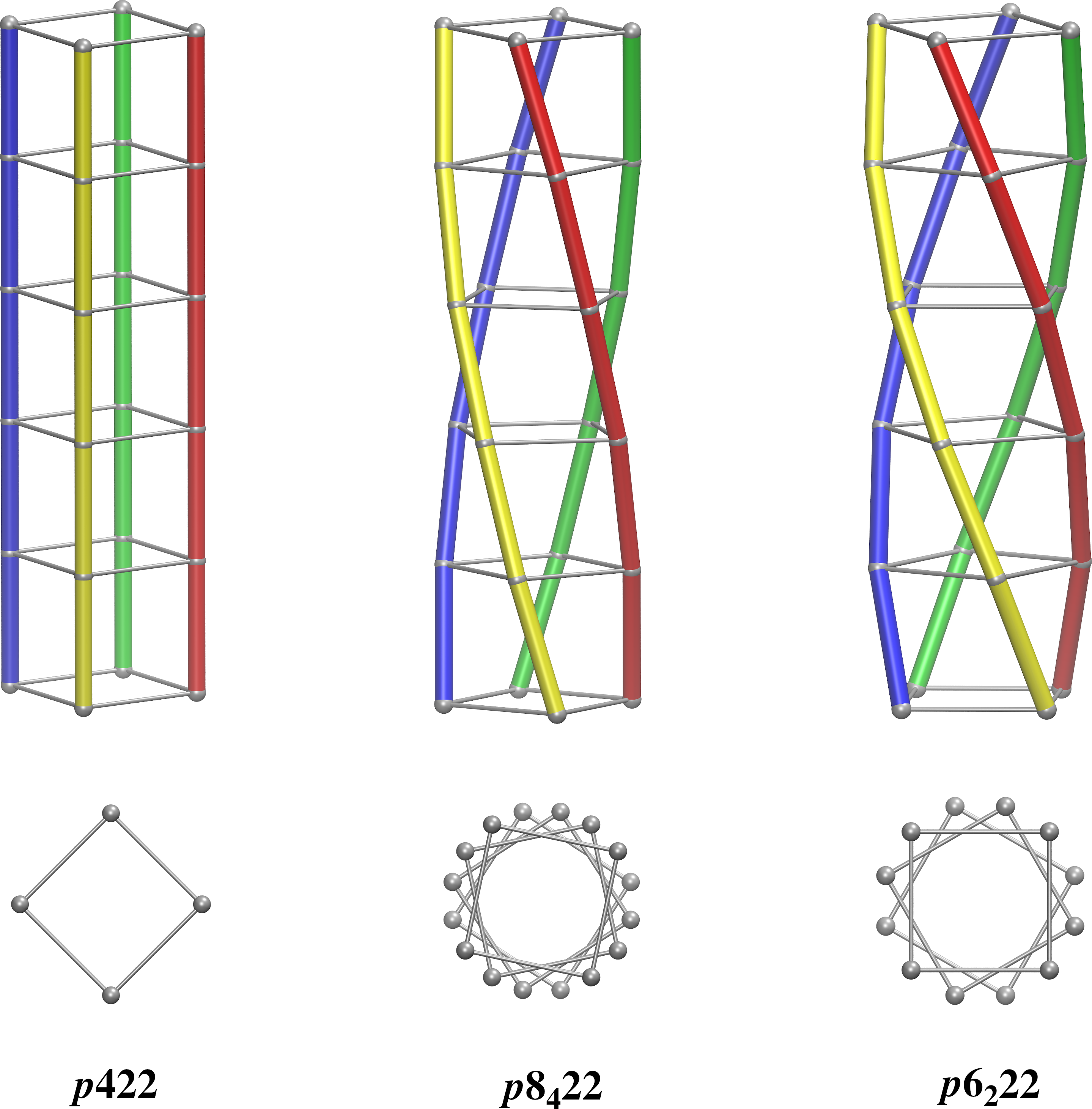}
\caption{Embeddings of the $\bf 4^4(0, 4)$ graph with different symmetry (side and top view).}
\label{f:3tubes}
\vspace{-2em}
\end{figure}
\end{center}

\noindent \emph{Remark 2}. Let us explain how the hexagonal group  $p6_222$ may appear as the symmetry group of the `intrinsically tetragonal' graph $\bf 4^4(0,4)$. Note that $p4_222$ is one of its regular groups. Owing to the isomorphism with $p4_222$, the group $p6_222$ is also plausible (Table~\ref{t:table1}). Consider the following group presentation:

\begin{center}
\begin{tabular}{l}
$G = \langle a, b, c, d | a^2, b^2, c^2, d^2, (cd)^2, dacb \rangle$; \\
$p4_222 = \langle 2(xx\frac{1}{4}), 2(xx\frac{\bar1}{4}), 2(x00), 2(0y0)\rangle$; \\
$p6_222 = \langle 2(x00), 2(x\bar{x}\frac{1}{3}), 2(0y\frac{1}{6}), 2(2x, x, \frac{1}{6})\rangle$. 
\end{tabular}
\end{center}

\noindent In the abstract group the distinction between $6_2(00z)$ and $4_2(00z)$ is lost, since both correspond to the element $ad$ of infinite order. Further, $H = \langle a, b\rangle$ is the \emph{translationengleiche} subgroup of $p4_222$ with index~4 ($H \cong p211$ and therefore it corresponds to disjoint `vertical edges'), whereas that of $p6_222$, $H \cong p6_522$, is the \emph{klassengleiche} subgroup of index~4. Hence, the embeddings with symmetry $p8_422$ and $p6_222$ are clearly isotopic (Fig.~\ref{f:3tubes}), so are those of symmetry $p422$ and $p4_222$.

{\tiny
\begin{table}[H]
\begin{center}
\caption{Isomorphism classes of rod groups}
\label{t:table1}
\begin{tabular}{|l|c|l|}
\hline
Groups & Direct product & FC-center \\
\hline
$p1$, $pc11$, $p112_1$, $p4_{1,3}$, $p3_{1,2}$, $p6_{1,5}$ & $C_{\infty}$ & $C_{\infty}$ \\
\hline
$p\bar1, p211, p2/c11, p11m, p112_1/m$  & $D_{\infty}$ & $C_{\infty}$ \\
$p222_1, p2cm, p4_{1,3}22, p3_{1,2}21, p6_{1,5}22$ & & \\
\hline
$pm11$, $p112$, $pcc2$, $pmc2_1$, $p4_2$, $p6_{2,4}$ & $C_{\infty} \times C_2$ & $C_{\infty} \times C_2$ \\
\hline
$p2/m11, p112/m, p222, p2mm$ & $D_{\infty} \times C_2$ & $C_{\infty} \times C_2$ \\
$pccm, pmcm, p4_222, p6_{2,4}22$ & & \\
\hline
$p3, p6_3$ & $C_{\infty} \times C_3$ & $C_{\infty} \times C_3$\\
\hline
$p4$ & $C_{\infty} \times C_4$ & $C_{\infty} \times C_4$ \\
\hline
$p6$ & $C_{\infty} \times C_6$ & $C_{\infty} \times C_6$\\
\hline
$p3m1$, $p6_3mc$ & $C_{\infty} \times D_3$ & $C_{\infty} \times D_3$ \\
\hline
$p\bar3$, $p\bar6$, $p6_3/m$ & $D_{\infty} \times C_3$ & $C_{\infty} \times C_3$\\
\hline
$p4/m$ & $D_{\infty} \times C_4$ & $C_{\infty} \times C_4$ \\
\hline
$p6/m$ & $D_{\infty} \times C_6$ & $C_{\infty} \times C_6$ \\
\hline
$p\bar3m1$, $p\bar6m2$, $p6_3/mmc$ & $D_{\infty} \times D_3$ & $C_{\infty} \times D_3$ \\
\hline
$pmm2$ & $C_{\infty} \times D_2$ & $C_{\infty} \times D_2$ \\
\hline
$p4mm$ & $C_{\infty} \times D_4$ & $C_{\infty} \times D_4$ \\
\hline
$p6mm$ & $C_{\infty} \times D_6$ & $C_{\infty} \times D_6$ \\
\hline
$pmmm$ & $D_{\infty} \times D_2$ & $C_{\infty} \times D_2$ \\
\hline
$p4/mmm$ & $D_{\infty} \times D_4$ & $C_{\infty} \times D_4$ \\
\hline
$p6/mmm$ & $D_{\infty} \times D_6$ & $C_{\infty} \times D_6$\\
\hline
$p\bar4$ &  -- & $C_{\infty} \times C_2$ \\
\hline
$p4_2/m$, $p\bar42c$ & -- & $C_{\infty} \times C_2$ \\
\hline
$p\bar42m$ & -- & $C_{\infty} \times D_2$ \\
\hline
$p3c1$ & -- & $p3c1$ \\
\hline
$p4cc$ & -- & $p4cc$ \\
\hline
$p6cc$ & -- & $p6cc$ \\
\hline
$p4_2mc$ & -- & $p4_2mc$ \\
\hline
$p4_2/mmc$ & -- & $p4_2mc$ \\
\hline
$p\bar3c1$, $p\bar6c2$ & -- & $p3c1$ \\
\hline
$p321$, $p6_322$ & -- & $C_{\infty} \times C_3$ \\
\hline
$p422$ & -- & $C_{\infty} \times C_4$ \\
\hline
$p622$ & -- & $C_{\infty} \times C_6$ \\
\hline
$p4/mcc$ & -- & $p4cc$ \\
\hline
$p6/mcc$ & -- & $p6cc$ \\
\hline
\end{tabular}
\end{center}
\end{table}
}

\subsubsection{Rod groups which are FC-groups}


In Table~\ref{t:table1} rod groups are characterized by FC-centers. For a Cayley graph $\Gamma$ of a group~$G$ we have: $F(G) = G \cap B(\Gamma)$ where $F(G)$ is the FC-center of $G$ and $B(\Gamma)$ is the group of all bounded automorphisms. Some rod groups are their own FC-centers. Apart from abelian groups, infinite FC-groups are very rare in crystallography. A group $G$ is an FC-group if its \emph{commutator subgroup} $G' = [G, G]$ is finite: in this case every element $g \in G$ has indeed finitely many conjugates: $\{h^{-1}gh = g[g, h] \, | \, h \in G\}$ where $[g, h]=g^{-1}h^{-1}gh \in G'$ \cite{Tom84}. In geometric terms, symmetry elements of a rod group $G$ which is an FC-group must all be parallel to the translation direction (so that in algebraic terms we get $T(G) \leq Z(G)$), thus an infinite number of conjugates by translations is avoided for any $g \in G$. As a consequence, all such rod groups are polar, and for them finiteness of conjugacy classes of elements can also be seen from symmetry diagrams \cite{ITE}.

In Cayley graphs of rod groups  bounded automorphisms of finite order are realised by rotation(s) about the principal axis (assumed to be vertical in the standard description \cite{ITE}) and by vertical mirrors, the latter act (informally speaking) `in the directions of missing periodicity'. Bounded automorphisms of infinite order correspond to translations, screw rotations and glide reflections (that are necessarily vertical in rod groups).

Cayley graphs of FC-groups provide examples of graphs with \emph{vertex-transitive groups of bounded automorphisms} \cite{Trofimov83, Seif91}, and from this point of view they are generalizations of translation lattices. For example, Cayley graphs of FC-groups on one hand, and of $\Z^n$ on the other, have the following characteristic property in common. First observe that in Cayley graphs of translation lattices (or abelian groups in general) the set of elements $\{g \in G | d(x, gx) = 1 \}$ is the same for any vertex~$x$. In the more general case of FC-groups, a given generating set $S$ can be always extended so that a larger, though finite, generating set $\bar{S} = \{ g^{-1}sg| s \in S, g \in G \}$ is obtained that is closed under conjugation. In other words, the set $\{g \in G | d(x, gx) = 1 \} = \bar{S}$ becomes independent of vertex~$x$ in the respective Cayley graph, as in the abelian case.

\medskip

\noindent \emph{Example:} Consider the FC-group $G = p4cc = \langle s_1, s_2 \rangle$ where $s_1 = \{m_{xxz}|00\frac{1}{2}\}$ and $s_2 = \{m_{yz}|00\frac{1}{2}\}$. The Cayley graph for this generating set is a square lattice rolled along the vector $(4, 4)$. The conjugacy classes of elements correspond one-to-one to those of $G/T(G) \cong D_4$ (Fig.~\ref{f:sql44}). We have: $S = \{ s_1, s_2\}; \bar{S} = \tilde{S} \cup \tilde{S}^{-1}$, where $\tilde{S} = \{s_1, s_2, s_2s_1s_2, s_1s_2s_1 \}$. The Cayley graph for $\bar{S}$ has valency~8, and therefore does not correspond to a quotient of any planar vertex-transitive graph.

\begin{center}
\begin{figure}[ht]
\centering
\includegraphics{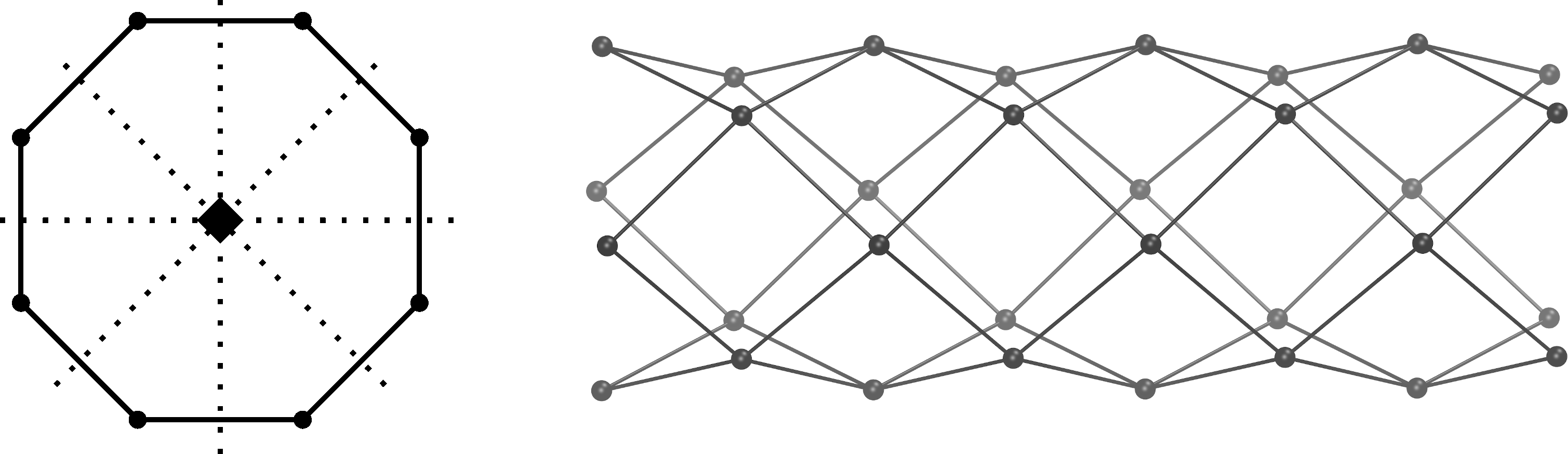}
\caption{Square lattice rolled along the vector $(4, 4)$ (top and side view). The top view is superposed with the symmetry elements of the rod group $p4cc$.}
\label{f:sql44}
\end{figure}
\vspace{-2em}
\end{center}

\subsection{Layer groups}

In Table~\ref{t:table2} layer groups are classified into 34 isomorphism types. This classification parallels that of two-dimensional space groups (\emph{wallpaper groups}). This correspondence becomes clear from the construction of layer groups as symmetry groups of a two-sided plane. Consider the $xy$-plane in $\R^3$ and a wallpaper group $H$ that maps the plane onto itself while fixing the orientation of the normal vector. In deriving a layer group $G$ from $H$, three cases are to be distinguished depending on the element $r$ exchanging the opposite sides of the plane \cite{Bur45}:

\begin{enumerate}
\setlength\itemsep{-0.02in}
\item[(a)] $r = 1$: $G = H$; 
\item[(b)] $r = m_{xy}$: $G = H \times \langle m_{xy} \rangle$; 
\item[(c)] $r = k \cdot m_{xy}, k \neq 1$: $G = H \cup (k \cdot m_{xy})H$ ($k \in K \setminus H$ for a wallpaper group $K > H, |K:H|=2$). 
\end{enumerate}

{\tiny
\begin{table}[t]
\begin{center}
\caption{Isomorphism classes of layer groups}
\label{t:table2}
\begin{tabular}{|l|l|l|l|}
\hline
Type & Groups & Type & Groups\\
\hline
$p1$ & $p1, p11b$ & $p1 \times C_2$ & $p11m$ \\
$p2$ & $p112, p\bar1, p112/b$ & $p2 \times C_2$ & $p112/m$ \\
$pm$ & $pm11, p121, p2_1ma, p2aa, p2mb, cm2a$ & $pm \times C_2$ & $pm2m$ \\
$pg$ & $p1a1, p12_11, p2_1ab$ & $pg \times C_2$ & $p2_1am$ \\
$cm$ & $cm11, c121, p2_1mn, p2an$ & $cm \times C_2$ & $cm2m$ \\
$p2mm$ & $pmm2, p12/m1, p222, pmaa, pmma, cmma$ & $p2mm \times C_2$ & $pmmm$ \\
$p2mg$ & $pma2, p12_1/m1, p12/a1, p2_122, pbaa, pmab$ & $p2mg \times C_2$ & $pmam$ \\
$p2gg$ & $pba2, p12_1/a1, p2_12_12$ & $p2gg \times C_2$ & $pbam$ \\
$c2mm$ & $cmm2, c2/m, c222, pban, pbmn, pmmn$ & $c2mm \times C_2$ & $cmmm$ \\
$p4$ & $p4, p\bar4, p4/n$ & $p4 \times C_2$ & $p4/m$ \\
$p4mm$ & $p4mm, p422, p\bar42m, p\bar4m2, p4/nbm, p4/nmm$ & $p4mm \times C_2$ & $p4/mmm$ \\
$p4gm$ & $p4bm, p42_12, p\bar42_1m, p\bar4b2$ & $p4gm \times C_2$ & $p4/mbm$ \\
$p3$ & $p3$ & $p3 \times C_2$ & $p\bar6$ \\
$p3m1$ & $p3m1, p312$ & $p3m1 \times C_2$ & $p\bar6m2$ \\
$p31m$ & $p31m, p321$ & $p31m \times C_2$ & $p\bar62m$ \\
$p6$ & $p6, p\bar3$ & $p6 \times C_2$ & $p6/m$ \\
$p6mm$ & $p6mm, p\bar3m1, p\bar31m, p622$ & $p6mm \times C_2$ & $p6/mmm$ \\
\hline
\end{tabular}
\end{center}
\vspace{-1em}
\end{table}
}

In group-theoretic terms, a layer group is either isomorphic to a wallpaper group [cases (a) and (c)], or is a direct product of a wallpaper group and $\langle m_{xy} \rangle$ [case (b)], hence the total number of types $34=17+17$. Note that in case~(b) $G$ is a central extension of $H$ by $\langle m_{xy} \rangle$, and in case~(c) $G$ is a subgroup of index~2 in $K \times \langle m_{xy} \rangle$ \cite{Koe80}. In case (c), group isomorphism between a layer group $G$ and a wallpaper group $K$ can be verified directly from group presentations as illustrated by Baburin~(2026) \cite{Bab26} and Souvignier (2026) \cite{Souv26} in connection with regular groups of the $6^3$-tessellation of the Euclidean plane.

Being in principle simple, the classification in Table~\ref{t:table2} seems not to have appeared in the literature before, or anyway the aspect of (abstract) group isomorphism was not emphasized. Some remarks on group isomorphism of subperiodic groups are contained in the papers by Fischer and Koch \cite{KF75, KF78b} who did not go, however, for a full classification beyond a few examples. Very recently, layer groups have been described in terms of orbifolds but again the aspect of group isomorphism was not considered \cite{layer25}.

The abelian groups in Table~\ref{t:table2} are $p1, p11b$ and $p11m$ which are also the only FC-groups among layer groups (the isomorphism between $p1$ and $p11b$ is noteworthy from a crystallographic point view). Since the horizontal mirror $m_{xy}$ is a central element, it induces the bounded automorphism of order~2 in respective periodic graphs \cite{Gromov83}.

As an application, let us point out the relevance of layer groups to non-isotopic embeddings of 2-periodic non-planar graphs in $\R^3$ (\emph{cf.} also \cite{Bab26}). Consider the {\bf{KIa}} graph \cite{KF78} that often arises in crystallography in different contexts. Its automorphism group, $Aut(\Gamma)$, is isomorphic to wallpaper group $p4mm$. By group isomorphism, and by considering also the vertex stabilizer, it can be shown that there are exactly two layer groups for its maximum-symmetry embeddings: $p\bar4m2$ and $p4/nbm$ (Fig.~\ref{f:klee}). The {\bf{KIa}} graph is built from two sets of chains which lie separately in parallel horizontal planes in the first case, and are interwoven in the other. Interestingly, the first ever occurrence of this graph in crystal structures (CSD code XAKCIV, sp.gr. $Pnna$) corresponds to the interwoven variant with symmetry $p2an$ (that is an index-4 subgroup of $p4/nbm$) (see Fig.~51 in \cite{Proser03}).

\begin{center}
\vspace{-1em}
\begin{figure}[t]
\centering
\includegraphics{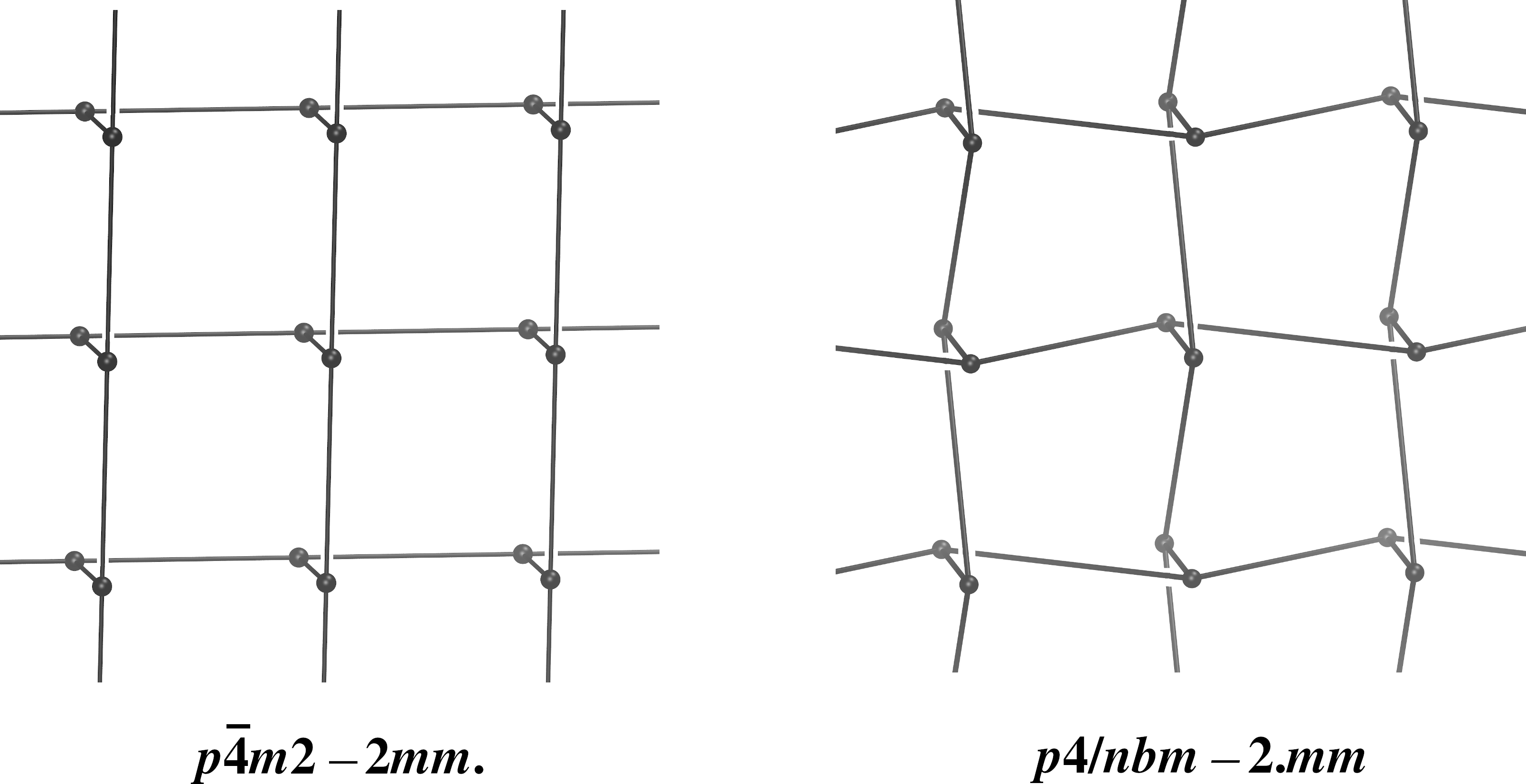}
\caption{Embeddings of the {\bf{KIa}} graph with different layer-group symmetry.}
\label{f:klee}
\end{figure}
\vspace{-2em}
\end{center}

\subsection{Invariants for crystallographic groups in $\R^3$} 

Following the ideas of the \textsc{Small Groups} library \cite{Besche2002} (as implemented in $\GAP$ \cite{GAP}), a practical approach for identifying groups consists in developing a set of invariants that is unique for a given isomorphism class, and a case-by-case consideration of groups sharing the same (finite sets of) invariants. It has to be emphasized that such invariants allow to distinguish groups within a given finite set. For crystallographic groups in $\R^3$, the invariants can be chosen particularly simple. Every group is characterized by the `fingerprint' that is an ordered list of 2-tuples [($m_1,n_1$), ... ($m_k, n_k$)]. The $kth$ tuple ($m_k, n_k$) refers to the number $m_k$ of conjugacy classes of subgroups of index $k+1$, and the number $n_k$ of normal subgroups of the same index. Experiment suggests that the maximal index to be considered is 12 for rod groups and 8 for layer groups, respectively, whereas for space groups in~$\R^3$, subgroups up to index~24 are needed. We emphasize that these invariants (deposited in the Supporting information) can be efficiently obtained from finite presentations using standard tools of computational group theory. The approaches for identifying crystallographic groups proposed in the literature are usually oriented towards space groups in matrix form and involve a more elaborate case-by-case analysis of the group structure (as done for wallpaper groups in \cite{Jones11}) or an explicit computation of a conjugating matrix given the `reference' list of groups \cite{Cryst2019}. 

The importance of considering normal subgroups can be clarified by the following example. Wallpaper groups $cm$ and $p2gg$ have the same number of conjugacy classes of subgroups up to index~11. Taking into account the number of normal subgroups allows to distiguish them starting from the index-3 subgroups:

\begin{center}
\begin{tabular}{l}
$cm: \, \, \, \, [ (3, 3), (2, 1), (7, 3), (2, 1), (7, 4), (2, 1), (13, 5), ... ]$;\\
$p2gg:  [ (3, 3), (2, 0), (7, 3), (2, 0), (7, 2), (2, 0), (13, 3), ... ]$.
\end{tabular}
\end{center}

\section{Automorphisms of Cayley graphs from group actions}

\subsection{Theory} \label{auts}

We begin with the fundamental result on Cayley graph automorphisms induced by group automorphisms.

\begin{prop}
\label{prop1}
\emph{(Biggs, 1993, ch.~16)}. Let $\Gamma$ be a Cayley graph of a group $G$ with respect to a generating set $S=S^{-1}$. An element $\bar{h} \in Aut(G)$ such that $\bar{h}(S) = S$ (i.e. $\bar{h}$ stabilizes $S$ setwise) induces an automorphism $h$ of $\Gamma$ fixing the vertex representing the identity of $G$.
\end{prop}

\noindent \emph{Remark}. Let $G$ be a regular group of automorphisms of $\Gamma$, and let $A(v) =  \{ sv|s \in S \}$ be the neighbourhood of the identity vertex $v$ of $\Gamma$. Since $hv=v$, besides an edge $\{v, sv\} = \{hv, shv\}$, there is also an edge $ \{v, (h^{-1}sh)v \}$, \emph{i.e.} $S= h^{-1}Sh$. Similarly, if $u = gv$ for some $g \in G$, $A(u) = \{ gsv = (gsg^{-1})u | s \in S \}$. 


\medskip

Let us emphasize that already the action of $G$ on itself by conjugation (\emph{i.e.} by inner automorphisms) allows us to find additional automorphisms of $\Gamma$ which are not left shifts of $G$ (central elements give rise to trivial automorphisms). This is because a non-trivial subgroup of $G$ can stabilize $S$ by conjugation, although $G$ is regular on vertices of~$\Gamma$ (Table~\ref{t:table3}). Everywhere in the following the stabilizer of $S$ refers to an inverse-closed generating set ($S = S^{-1}$).

{\tiny
\begin{table}[H]
\caption{Stabilizers of generating sets for some Cayley graphs}
\label{t:table3}
\vspace{0.3em}
\small \raggedright \hspace{3.2em} `--' in the last column refers to trivial $Z(G)$.\\
\vspace{-1em}
\begin{center}
\begin{tabular}{|l|l|l|l|c|}
\hline
Graph & Generating set $S$ & $G$ & $Stab_G(S \cup S^{-1})$ & $Stab = Z(G)$ \\
\hline
$K_{3,3}$  &  $(12)(34)(56), (14)(25)(36),$ & $D_3$ & $D_3$ & -- \\
           &   $(16)(23)(45)$               &       &       & \\
\hline
\rule{0pt}{2.5ex} cube       & $4(00z), 2(x00)$ & $D_4$ & $D_2 = \langle 2(x00), 2(0y0) \rangle$ & no \\
\hline
\rule{0pt}{2.5ex} rhombocuboctahedron & $3(xxx), m_{xz}, m_{yz}$ & $m\bar3$ & $C_i = \langle \bar1(000) \rangle$ & yes \\
& $4(00z), 3(xxx)$ & $432$ & $C_2 = \langle 2(x\bar{x}0)\rangle$ & -- \\
\hline
\rule{0pt}{2.5ex} truncated tetrahedron & $3(xxx), 2(00z)$ & $23$ & $C_1$ & -- \\
\hline
\rule{0pt}{2.5ex} icosahedron & $3(xxx), 3(xx\bar{x}), 2(x00)$ & $23$ & $C_2 = \langle 2(00z)\rangle$ & -- \\
\hline
\rule{0pt}{2.5ex} $\bf 4^4(0, 3)$ & $3(00z), 2(x00), 2(x0\frac{1}{2})$ & $p321$ & $C_1$ & -- \\
\hline
\rule{0pt}{2.5ex} $\bf 4^4(0, 4)$ & $\bar4_z(00\frac{1}{4}), 2(x00), 2(x0\frac{1}{2})$ & $p\bar42c$ & $C_2 = \langle 2(00z)\rangle$ & yes \\
\hline
\rule{0pt}{2.5ex} $3/10/c1$ ({\bf srs}) & $4_3(\frac{1}{4}\frac{1}{4}z), 2(00z)$ & $I4_1$ & $p1 = \langle t_z \rangle$ & yes \\
\hline
\rule{0pt}{2.5ex} $4/6/c1$ ({\bf dia}) & $\bar1(000), 2(0y\frac{1}{4}), b(\frac{1}{4}yz)$  & $Pbcn$ & $C_1$ & -- \\
\hline
\rule{0pt}{2.5ex} $4/6/t4$ ({\bf cds}) & $4_3(\frac{1}{4}\frac{1}{4}z), b(\frac{1}{4}yz)$  & $I4_1cd$ & $p2_1 = \langle 2_1(\frac{1}{4}\frac{1}{4}z) \rangle$ & no \\
\hline
\rule{0pt}{2.5ex} $6/4/c1$ ({\bf pcu}) & $4_1(00z), t_x, t_y$  & $P4_1$ & $p4_1 = \langle 4_1(00z) \rangle$ & no \\
\rule{0pt}{2.5ex} &  $4(00z), 4(\frac{1}{2}\frac{1}{2}z), c(xxz)$ & $P4cc$ & $pc = \langle c(xxz) \rangle$ & no \\
\hline
\end{tabular}
\end{center}
\end{table}
}


\begin{center}
\begin{figure}[ht]
\centering
\includegraphics{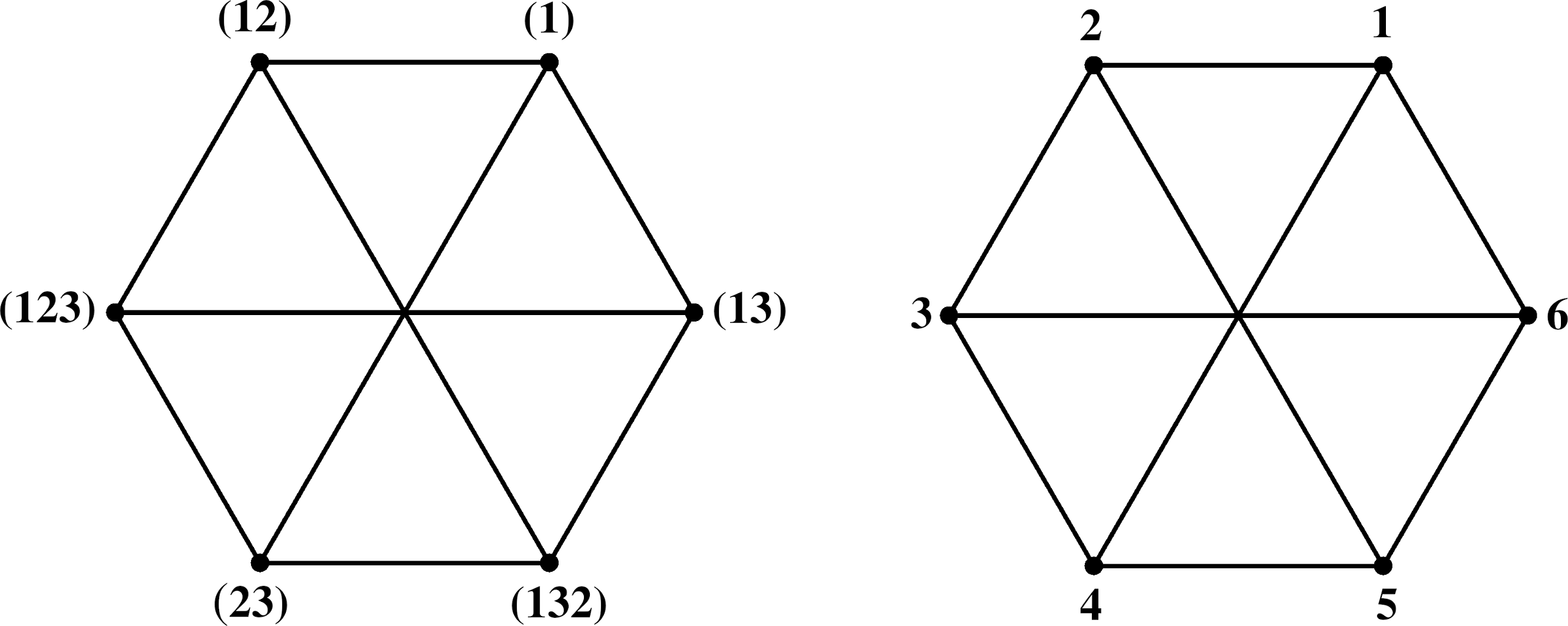}
\caption{$K_{3,3}$ as a Cayley graph for $D_3$.}
\label{f:K33}
\end{figure}
\vspace{-2em}
\end{center}

\noindent \emph{Examples}: (a) Following Biggs (1993), construct a Cayley graph $\Gamma$ for $G = D_3 = Sym(3)$ given the generating set of three involutions $\{ (12), (13), (23) \}$. Let us determine automorphisms of $\Gamma$ (that is isomorphic to $K_{3,3}$ -- see Fig.~\ref{f:K33}) from the action of $G$ on itself by left shifts. In doing so, we identify the vertices $v_i$ of $\Gamma$ (as numbered in Fig.~\ref{f:K33}) with the elements of $G$:

\begin{center}
\begin{tabular}{l}
$( \, (1), (12),  (123), (23), (132), (13) \, ) \to (v_1, v_2, v_3, v_4, v_5, v_6)$.
\end{tabular}
\end{center}

\noindent Left shifts by the transpositions $(12), (13)$ and $(23)$ yield the following results:

\begin{center}
\begin{tabular}{l}
$(v_1, v_2, v_3, v_4, v_5, v_6) \to (v_2, v_1, v_4, v_3, v_6, v_5)$;\\
$(v_1, v_2, v_3, v_4, v_5, v_6) \to (v_6, v_3, v_2, v_5, v_4, v_1)$;\\
$(v_1, v_2, v_3, v_4, v_5, v_6) \to (v_4, v_5, v_6, v_1, v_2, v_3)$.\\
\end{tabular}
\end{center}

\noindent The obtained permutations generate the regular group of automorphisms of $\Gamma$ that is isomorphic to $D_3$:

\begin{center}
\begin{tabular}{l}
$\bar{G} = \langle (v_1v_2)(v_3v_4)(v_5v_6), (v_1v_6)(v_2v_3)(v_4v_5), (v_1v_4)(v_2v_5)(v_3v_6) \rangle$.
\end{tabular}
\end{center}

\noindent Now determine additional automorphisms of $\Gamma$ from the conjugation of $G$ by the transpositions $(12)$ and $(13)$:

\begin{center}
\begin{tabular}{l}
$(v_1, v_2, v_3, v_4, v_5, v_6) \to (v_1, v_2, v_5, v_6, v_3, v_4)$;\\
$(v_1, v_2, v_3, v_4, v_5, v_6) \to (v_1, v_4, v_5, v_2, v_3, v_6)$.\\
\end{tabular}
\end{center}

\noindent As a result, we have found a larger automorphism group $\ti{G} \cong D_3 \times D_3$:

\begin{center}
\begin{tabular}{l}
$\ti{G} = \langle \bar{G}, (v_3v_5)(v_4v_6), (v_2v_4)(v_3v_5) \rangle$,
\end{tabular}
\end{center}

\noindent together with the stabilizer of $v_1$, $\ti{G}_{v_1} = \langle (v_3v_5)(v_4v_6), (v_2v_4)(v_3v_5) \rangle \cong D_3$. Note that $|Aut(\Gamma)|=72$ and $Aut(\Gamma)_{v_1} = \langle \ti{G}_{v_1}, (v_2v_6) \rangle \cong D_6$.

\medskip

\noindent (b) Generating sets of polar space groups in Table~\ref{t:table3} are characterised by infinite stabilizers which are rod groups (in this case $Z(G) \neq \{1\}$). Consider the $6/4/c1$ ({\bf pcu}) graph that is isomorphic to the three-dimensional cubic lattice. The fourfold right-handed screw rotation $f$ induces the following permutation of generators:

\begin{center}
\begin{tabular}{l}
$(f, f^{-1}, t_x, t_{x}^{-1}, t_y, t_y^{-1}) \to (f, f^{-1},  t_y, t_y^{-1} , t_{x}^{-1}, t_x)$.\\
\end{tabular}
\end{center}

\noindent This permutation induces an automorphism of the graph that is realised by the fourfold rotation $4(00z)$ that leads to the enhancement of the automorphism group from $P4_1$ to $P4$.

\medskip

For an arbitrary group there can be different `mechanisms' for the stabilization of generating sets~\cite{Watk71} (\emph{e.g.} if a generating set is closed under conjugation, it is clearly stabilized by the whole group, as in the above case of $K_{3,3}$). For the examples in Table~\ref{t:table3}, the following observation holds: for a regular group $G = \langle S \rangle$ of automorphisms of a graph~$\Gamma$, any element (of $G$) that nontrivially stabilizes $S$ is of the form $hz$, where $h \in Aut(\Gamma)_x$ is  such that $h^{-1}Sh=S$, and $z \in C_{Aut(\Gamma)}(G)$. The next proposition puts this observation into a slightly more general context.

\begin{prop}
\label{prop2}
Let $H<G$ be vertex-transitive groups of automorphisms of a connected graph~$\Gamma$. For any $x \in V(\Gamma)$, put $A_G(x) = \{g \in G | d(x, gx) = 1\}$. Then $\langle G_x, Z(G) \rangle \leq Stab_G(A_G(x))$ and  $\langle Stab_G(A_G(x)) \cap N_G(H), C_G(H) \rangle \leq Stab_G(A_H(x))$.
\end{prop}
\begin{proof}
From \cite[Theorem~2]{Sabidussi64} we have $G = \langle A_G(x) \rangle$ and ${g_x}^{-1}A_G(x)g_x = A_G(x)$ for any $g_x \in G_x$; it easily follows that $\langle G_x, Z(G) \rangle \leq Stab_G(A_G(x))$. Clearly $A_H(x) \subset A_G(x)$, $H = \langle A_H(x) \rangle$ and $Stab_G(A_H(x)) \leq N_G(H)$. The result is immediate.
\end{proof}

\noindent \emph{Example}: As an application, Table~\ref{t:table4} lists the stabilizers (in $Aut(\Gamma)$) of generating sets for regular groups of automorphisms of the diamond graph (\emph{cf.} Table~1 in \cite{Bab26}). As expected from Proposition~\ref{prop2}, stabilizers are finite for centerless regular groups (\emph{cf.} Proposition~\ref{prop4} below), and are symmorphic subperiodic groups otherwise. The most `asymmetric' generating set occurs for space group $P3_121$. Nontrivial stabilizers might be viewed as local indicators of possible pseudosymmetry.

{\tiny
\begin{table}[ht]
\caption{Stabilizers of generating sets for regular groups of the diamond graph}
\label{t:table4}
\vspace{0.3em}
\small \raggedright \hspace{4em} $Aut(\Gamma) =  \langle a, b, c \rangle = Fd\bar3m; a = (y, -x, -z); b = (y, z, x);  c = (1/4-x, 1/4-y, 1/4-z)$.
\vspace{-1em}
\begin{center}
\begin{tabular}{|c|l|c|c|c|}
\hline
Group & Generating set $S$ & \multicolumn{3}{c|}{$Stab_{Aut(\Gamma)}(S \cup S^{-1})$}    \\
\cline{3-5}
      &                    &  Generators                       &    IT/Schoenﬂies       & Translations          \\ 
\hline
\rule{0pt}{2.5ex} $P\bar1$ & $c, aca^{-1}, a^{-1}ca, a^2ca^2$  & $ a, b$  & $T_d$ & -- \\
\rule{0pt}{2.5ex} $Cc$  &  $bcaba, a^{-1}bcab$ & $  a^2, ba, aca^{-1}c, (ac)^4$ & $p\bar4m2$ & $(\frac{1}{2} 0 \frac{1}{2}), (001)$ \\ 
\rule{0pt}{2.5ex} $P222_1$ & $a^2cab, a^2b^{-1}a^{-1}c, abaca, aca^2b $ & $ ab, (ba)^2$ & $D_{2d}$ & -- \\ 
\rule{0pt}{2.5ex} $P2_12_12_1$ & $ ca^{-1}b, abca^2$ & $''$ & $''$ & $''$ \\
\rule{0pt}{2.5ex} $P4_1$  &  $ca^{-1}, a^{-1}ca^2$ &  $ a^2, (ab)^2, (ac)^4$ & $p222$ & $(001)$ \\
\rule{0pt}{2.5ex} $P4_12_12$ & $a^2cab, a^2b^{-1}a^{-1}c, abaca^{-1}$ & $a^2$ & $C_2$ & -- \\
\rule{0pt}{2.5ex} $C2$  &  $ca^{-1}b, a^2cab, aca^2b$ & $ a^{-1}b, (ab)^2,  a^{-1}cac$ & $p2mm$ & $(0\frac{1}{2}\frac{1}{2})$ \\ 
\rule{0pt}{2.5ex} $Pnc2$ & $a^2cab, aca^2b, bca^2b^{-1}$ & $''$ & $''$ & $''$ \\
\rule{0pt}{2.5ex} $Pna2_1$ & $ca^{-1}b, abcab$ & $''$ & $''$ & $''$ \\
\rule{0pt}{2.5ex} $P2_1/c$ & $aca^{-1}, ca^{-1}b, a^2ca^2$ & $ a^{-1}b, (ab)^2$ & $C_{2v}$ & -- \\
\rule{0pt}{2.5ex} $P2/c$   & $c, a^{-1}ca, a^2cab, aca^2b$ & $''$ & $''$ & $''$ \\
\rule{0pt}{2.5ex} $P2_1/c$  & $aca^{-1}, a^2ca^2, bca^2b^{-1}$ & $''$ & $''$ & $''$ \\ 
\rule{0pt}{2.5ex} $Pbcn$ & $a^2ca^2, aca^2b, bca^2b^{-1}$ & $  a^{-1}b$ & $C_s$ & -- \\
\rule{0pt}{2.5ex} $Pnna$ & $ a^2ca^2, a^2b^{-1}a^{-1}c, abaca, aca^2b$ & $''$ & $''$ & $''$ \\
\rule{0pt}{2.5ex} $C2/c$ & $ ca^{-1}b, a^2ca^2, aca^2b $ & $''$ & $''$ & $''$ \\
\rule{0pt}{2.5ex} $C2/c$ & $ c, aca^{-1}, a^{-1}ca, a^2cab $ &  $''$ & $''$ & $''$ \\
\rule{0pt}{2.5ex} $P3_121$ & $abaca, abca^2, a(ab)^2c$ & $1$ & $C_1$ & -- \\
\hline
\end{tabular}
\end{center}
\end{table}
}

\bigskip

\noindent We proceed with a general result on the stabilizer of a generating set of a finitely generated group.

\begin{prop}
\label{prop3}
Let $S = S^{-1} \subset G$ be a finite set of generators of a group $G$. Then the group $H = Stab_G(S)$ is an FC-group.
\end{prop}

\begin{proof} By assumption, $G=\langle S \rangle$, and hence $C_G(S) = Z(G) \trianglelefteq H$. Since $Z(G)$ is abelian (possibly trivial) and clearly $Z(G) \leq Z(H)$, then, because $S$ is finite, the group $H$ is an extension of $Z(G)$ by a finite group. The latter implies $H$ to be an FC-group \cite[Corollary~1.5]{Tom84}.
\end{proof}

\begin{prop}
\label{prop4}
Let $\Gamma$ be a Cayley graph of a space group $G$ with respect to a generating set $S=S^{-1}$, and let $Aut(\Gamma)$ be some space group. Then any $g \in G$ such that $S = g^{-1}Sg$ is either the identity or an element of infinite order. Moreover, $g = hz = zh$ for some $h \in Aut(\Gamma)_{1_G}$ and some $z \in C_{Aut(\Gamma)}(G)$.
\end{prop}

\begin{proof} Assume $g \neq 1$ and $g \notin T(G)$ (for, if $g \in T(G)$, then $g \in Z(G)$), \emph{i.e.} $G \lneq Aut(\Gamma)$. By Proposition~\ref{prop1}, $g$ induces an automorphism $h \in Aut(\Gamma)_{1_G}$ such that $S = h^{-1}Sh$. The linear part of~$h$ is determined by the conjugation action of $g$ on $T(G)$ that guarantees for both $g$ and $h$ to have the same linear part (since $Aut(\Gamma)_{1_G}$ is a finite group, the translation part of $h$ can be set to zero), and we may write $g = hz$ for some $1 \neq z \in T(Aut(\Gamma))$. For $hz \in G$ of finite order, we have $S \neq (hz)^{-1}S(hz)$. We are left with $hz$ of infinite order\footnote{Examples of Cayley graphs with symmetry $I4_1cd, P4_1, P4cc$ in Table~\ref{t:table3} clarify this point: speaking informally, the action of a screw (glide) rotation (reflection) mimics that of a pure rotation (reflection) since the shift by a screw (glide) vector centralizes the group.}, in which case $z \in C_{Aut(\Gamma)}(h)$. If $hz$ stabilizes $S$, then so does $(hz)^n = z^n, n = |h|$. Therefore $z^n \in Z(G)$, and by the Remark below, $z \in C_{Aut(\Gamma)}(G)$. By the regularity of $G$, $z \notin Z(G)$ for $h \neq 1$, and if $C_{Aut(\Gamma)}(G) = \{1\}$, $hz = 1$.
\end{proof}

\medskip

\noindent \emph{Remark}. Let $G$ and $K$ be space groups of $\R^d$, and let $G < K$. The centralizer of $G$ in $A=\Aff(\R^d)$ (the group of all affine mappings of $\R^d$) is $C_A(G) \subseteq \R^d$, \emph{i.e.} the group of continuous translations in $r \leq d$ linearly independent directions (possibly $r=0$) \cite[Lemmas~6 and 8]{Rat15}. All centralizing translations span an $r$-dimensional subspace of $\R^d$ that is invariant under the point group $P(G)$  \cite[Lemma~5]{Rat15}, and the dimensions of both $Z(G)$ and $C_A(G)$ -- regarded as vector spaces (over reals) -- coincide. Clearly $Z(G) = C_A(G) \cap G$ and $C_K(G) = C_A(G) \cap K$.

\begin{prop}
\label{prop5}
Under the assumptions of Proposition~\ref{prop4}, $Stab_G(S)$ is torsion-free abelian.
\end{prop}

\begin{proof} By Proposition~\ref{prop3}, $Stab_G(S)$ is an FC-group. Further, by Proposition~\ref{prop4}, it has no nontrivial finite-order elements. Then it is torsion-free abelian.
\end{proof}

\bigskip

\noindent The following theorem provides a sufficient condition when a Cayley graph of a space group has nontrivial bounded automorphisms of finite order. 

\begin{theorem}
\label{teo1}
Let $\Gamma$ be a Cayley graph of a space group $G$ with respect to a (finite) generating set $S=S^{-1}$. If $S = g^{-1}Sg$ for some $1 \neq g \in G$ of finite order, then $\Gamma$ has nontrivial bounded automorphisms of finite order.
\end{theorem}

\begin{proof} Recall that vertices $x, y \in V(\Gamma) = G$ are connected by an edge, if and only if $x^{-1}y \in S$. There is also an edge between vertices $xg, yg$ since $(xg)^{-1}yg = g^{-1}(x^{-1}y)g \in S$. Therefore, the right shift by $g$, \emph{i.e.} the map $x \mapsto xg$ for all $x \in V(\Gamma)$, is an automorphism of $\Gamma$ (of order $|g|$) acting freely on $V(\Gamma)$. To show that this automorphism is bounded, observe that $d(x, xg) = d(1_G, g)$ that is clearly a constant.\footnote{In the present form Theorem~\ref{teo1} appeared thanks to a hint from a referee.}
\end{proof}

\noindent \emph{Remark 1}. The proof remains valid for a locally finite Cayley graph $\Gamma$ of a finitely generated group~$G$. 

\medskip

\noindent \emph{Remark 2}. If $S = g^{-1}Sg$ for $g \in G$ of infinite order, the map $x \mapsto xg$ is an automorphism of $\Gamma$ that is a nontrivial translation. (Let $g=\{g_l|g_t\}$ where $g_l$ and $g_t$ are the linear and translation parts of $g$, resp. From Proposition~\ref{prop4} we get $\{g_l|g_o\} \in Aut(\Gamma)_{1_G}$, so the right action of $g$ induces the translation by $(g_t - g_o)$).

\medskip

\noindent \emph{Remark 3}. Bounded automorphisms of finite order constitute a normal subgroup $B_0(\Gamma)$ of $Aut(\Gamma)$ \cite{Trofimov84, Seif89all}. $\Gamma$ has polynomial growth since $T(G) \cong \Z^d$ acts on $V(\Gamma)$ with finitely many orbits (\emph{cf.} Section~2). From the characterization of graphs with polynomial growth \cite{Trofimov84} it follows that $B_0(\Gamma)$ has finite orbits on $V(\Gamma)$ which partition $V(\Gamma)$ into (finite) blocks (for more details see Appendix~A). Therefore, the generating set $S$ of $G$ can be viewed as a union of the two subsets $S = S_i \cup S_e$ corresponding to the \emph{intra}-block ($S_i$) and \emph{inter}-block ($S_e$) edges, respectively. The distinction between two kinds of edges allows for a simple construction of Cayley graphs having bounded automorphisms of finite order for the generating sets which are trivially stabilized as shown below.

\medskip

\noindent \emph{Examples}. (a) Consider the generating set $S = S_i \cup S_e$ for space group $G=P23$ corresponding to the arrangement of truncated tetrahedra at the nodes of the primitive cubic lattice:

\begin{center}
\begin{tabular}{l}
$S_i = \{3^{+}(xxx), 3^{-}(xxx), 2(00z)\}$,\\
$S_e = \{gt | g \in \langle S_i \rangle, t \in \{t_x, t_x^{-1}, t_y,  t_y^{-1}, t_z, t_z^{-1}\}\}$.
\end{tabular}
\end{center}

\noindent The respective Cayley graph has valency 75. By construction, it admits bounded automorphisms of finite order which, for example, permute vertices in one of the blocks while fixing all the other blocks pointwise. We have: $Stab_G(S_i) = \{1\}$ (\emph{cf.} Table~\ref{t:table3}) and $Stab_G(S_e) = \langle S_i \rangle \cong A_4$. Therefore $Stab_G(S) = \{1\}$.

\medskip

\noindent (b) Now show that in general the subset $S_e$ need not be nontrivially stabilized. The set $S = S_i \cup S_e$, with

\begin{center}
\begin{tabular}{l}
$S_i = \{ 3^{+}(00z), 3^{-}(00z) \}$,\\
$S_e = \{ 2(x00), 2(x\bar{x}\frac{1}{4}), 2(2x, x, \frac{1}{4}), t_x, t_x^{-1}, t_y,  t_y^{-1}, t_xt_y, t_x^{-1}t_y^{-1} \}$,
\end{tabular}
\end{center}

\noindent generates space group $G=P6_322$. We have: $2(x00) \in Stab_G(S_i) \setminus C_G(S_i)$ and $Stab_G(S_e) = \{1\}$. The Cayley graph can be imagined as triangles centered at the nodes of the primitive hexagonal lattice. It can be verified that the graph has a bounded automorphism of order~2 acting as a local mirror on triangular blocks: it centralizes the translation lattice $T(G)$ and, within each block, fixes one vertex and exchanges the other two. Note that this automorphism is retained if triangular blocks become disconnected, \emph{i.e.} when $S = S_e$ since $\langle S_e \rangle = P6_322$.

\medskip

Nevertheless, bounded automorphisms of finite order of a Cayley graph always give rise to the nontrivial stabilizer of a generating set in the following frequent special case.

\begin{corollary}
\label{cor1}
If $B_0(\Gamma)$-orbits on $V(\Gamma)$ are all of length~2, Theorem~\ref{teo1} supplies a necessary and sufficient condition for $\Gamma$ to have nontrivial bounded automorphisms of finite order.
\end{corollary}

\begin{proof}  Let $\sigma$ be the partition of $V(\Gamma)$ into $B_0(\Gamma)$-orbits (blocks), and let $x^{\sigma}$ be the block of~$\sigma$ containing vertex $x$. By assumption, $|x^{\sigma}|=2$ and, by the arguments from \cite[Proposition~4 and Theorem~1]{Trofimov2014}, there exists $z \in Z(Aut(\Gamma))$ of order~2 exchanging vertices in all the blocks simultaneously. For a vertex-transitive group $H \leq Aut(\Gamma)$ and any $x \in V(\Gamma)$, define the set $A_H(x) = \{h \in H|d(x, hx)=1\}$. Consider an arbitrary block $x^{\sigma} = \{x, y\}$. Since $y=zx$, it follows that $A_{Aut(\Gamma)}(y) = z^{-1}A_{Aut(\Gamma)}(x)z$ and, because $z \in Z(Aut(\Gamma))$, $A_{Aut(\Gamma)}(x) = A_{Aut(\Gamma)}(y)$ and therefore $A_G(x) = A_G(y)$. For the block $\{1_G, g\}$ (clearly $|g|=2$), we have $A_{G}(1_G) = S = g^{-1}Sg = A_G(g)$.
\end{proof}

\begin{corollary}
\label{cor2}
Let $H  = \langle g \in G | g^{-1}Sg = S, |g| < \infty \rangle$ and let $\sigma$ be the partition of $V(\Gamma)$ into $B_0(\Gamma)$-orbits (blocks). For the block $x^{\sigma}$ containing $1_G$, $H \subseteq x^{\sigma}$. 
\end{corollary}

\noindent Corollary~\ref{cor2} states essentially that $\sigma$ can be partially `reconstructed' from the stabilizer of a generating set: the (generally finer) partition $\tau = \{gH | g \in G\}$ is contained in the partition $\sigma$ that can be written as $\tau \preceq \sigma$.

\medskip

\noindent The above results provide a recipe for constructing Cayley graphs of space groups which \emph{a priori} possess nontrivial bounded automorphisms of finite order.

\begin{corollary}
\label{cor3}
Let $G$ be a space group, and $H$ be a nontrivial finite subgroup of $G$. Let $S=S^{-1}$ be a finite set of generators for $G$ such that $S$ is a union of finitely many orbits of the conjugation action of $H$ on elements of~$G$. Then, by construction, $H \leq Stab_G(S)$ and the respective Cayley graph has nontrivial bounded automorphisms of finite order.
\end{corollary}

\noindent Moreover, the following general construction due to de la Salle and Tessera (2019) (which can be viewed as `symmetrization' of a generating set by a finite subgroup) leads to Cayley graphs with infinite vertex stabilizers in $Aut(\Gamma)$ (an example will be given in the next section).

\begin{theorem}
\label{teo2}
\emph{(cf. \cite[Lemma~6.1]{Salle19})}\footnote{De~la Salle and Tessera (2019) formulate the statement that (as they explain right after their Lemma~6.1) is equivalent to saying that $Aut(\Gamma)$ is not discrete. In the language of topological group theory this means that vertex stabilizers are infinite (indeed, uncountable) (\emph{cf.} \cite[Proposition~1]{Trofimov85}).} Let $G$ be a space group with a finite generating set $S=S^{-1}$ and $H$ be a nontrivial finite subgroup of $G$. The generating set $\{hsh'|h, h' \in H, s\in S \}$ gives rise to a Cayley graph $\Gamma$ with infinite vertex stabilizers in $Aut(\Gamma)$.
\end{theorem}

\noindent However, not all space groups allow for Cayley graphs with (nontrivial) bounded automorphisms of finite order.

\begin{prop} 
\label{prop6}
Let $G$ be a torsion-free space group (also termed a fixed-point-free space group or a Bieberbach group). Any Cayley graph of $G$ with respect to a finite generating set has no nontrivial bounded automorphisms of finite order.
\end{prop}

\begin{proof} If $B_0(\Gamma) \neq \{1\}$, then $B_0(\Gamma)$ has nontrivial finite orbits on $V(\Gamma)$. The stabilizer in $G$ of such an orbit is a nontrivial finite subgroup of $G$, a contradiction.
\end{proof}

\subsection{Cayley graphs of crystallographic groups: applications} 

The above theory has the following practical implications. Let $\Gamma$ be a Cayley graph of a space group $G$. The conjugation action of $G$ (or possibly of its affine normalizer) on $G$ can be applied to compute additional automorphisms of $\Gamma$ that complements and extends available methods. The latter include the method of barycentric placement \cite{Delgado2004} and the method of vertex stabilizers \cite{Bab20, Trofimov2012}. The method of barycentric placement can be applied when $Aut(\Gamma)$ is isomorphic to some space group, and furthermore, if neighbourhoods of vertices in a graph are injectively mapped to their barycentric positions (in terms of O.~Delgado such embedding is \emph{neighbour-unique}). The method \cite{Bab20} puts no restrictions on $Aut(\Gamma)$ and allows to compute vertex stabilizer in $Aut(\Gamma)$ as a permutation group from a sufficiently large finite ball that requires, however, more preprocessing steps. Infinite stabilizers (being not finitely generated \cite[Proposition~1]{Trofimov85}, \emph{cf.}~Appendix~A) are recognized from the order of the restriction of $Aut(\Gamma)_x$ to $B(x, r) = \{y \in V(\Gamma) |d(x, y) \leq r\}$ which gets arbitrarily large as soon as $r$ is increased.

The main problem in computing $Aut(\Gamma)$ consists in handling bounded automorphisms of finite order. It is known that if a periodic graph has nontrivial bounded automorphisms of finite order, then its barycentric placement shows up vertex-collisions \cite{Eon2013}. However, this statement cannot be reversed, and there exist periodic graphs with vertex-collisions but without nontrivial bounded automorphisms of finite order (\emph{cf.} examples below). Theorem~\ref{teo1} provides a sufficient criterion when $Aut(\Gamma)$ is \emph{not} a space group for a very broad class of vertex-transitive periodic graphs which are Cayley graphs of space groups. On the other hand, the collision-free barycentric placement ensures that $Aut(\Gamma)$ \emph{is} a space group \cite[Corollary~1]{Delgado2004}.

The examples in the next section are aimed at exhibiting certain aspects of Cayley graphs which are beyond geometric intuition. To facilitate the discussion, the following uniform notation is introduced: $\Gamma$ is a Cayley graph of a group $G$ for a generating set $S$, $B(\Gamma) \trianglelefteq Aut(\Gamma)$ is the group of all bounded automorphisms of $\Gamma$, and $B_0(\Gamma) \vartriangleleft Aut(\Gamma)$ is the group of all bounded automorphisms of finite order \cite{Seif89all}. Sometimes we shall also use the notation $G_0 = G \cap B_0(\Gamma)$. For vertex stabilizers the notation $Aut(\Gamma)_v$ is preferred over $Aut(\Gamma)_x$ to avoid confusion with coordinate axes. Recall that $Stab_G(S)$ always refers to $Stab_G(S \cup S^{-1})$.

\subsubsection{Examples}

\noindent \emph{1-Periodic graphs}. The tabulation of homogeneous sphere packings with rod-group symmetry by Koch and Fischer (1978) provides a rich collection of Cayley graphs. Therein, four types [$\bf 3^6(0,2)$, $\bf 3^34^2(0,2)$, $\bf 4^4(2,2)$ and $\bf 6^3(0,2)$] stand out which have uncountable automorphism groups. They are all related to the graph $\bf 3^6(0,2)$ (that is the \emph{strong product} of the infinite chain and $K_2$) which can be visualized as a band of edge-sharing tetrahedra (Fig.~\ref{f:infstab}). This graph is known in the mathematical literature since 1980s as a classical example of a vertex-transitive graph with infinite vertex stabilizers \cite{Jung84, Seif91}. The reason for this is that vertices within the blocks $\{u_i, v_{i} \}$ (Fig.~\ref{f:infstab}) can be permuted independently. That this graph can also be viewed as a Cayley graph of the infinite dihedral group $D_{\infty}$ (one of the two one-dimensional space groups) does not seem to be well-known. Consider the minimal generating set for $D_{\infty}$:

\begin{center}
\begin{tabular}{l}
$D_{\infty} = \langle a, b | a^2, (ab)^2 \rangle$.\\
\end{tabular}
\end{center}

\noindent By Theorem~\ref{teo2}, it can be symmetrized with respect to $\langle a \rangle$ as follows:

\begin{center}
\begin{tabular}{l}
$D_{\infty} = \langle a, b, b^{-1}, ab, ba \rangle$.\\
\end{tabular}
\end{center}

\begin{center}
\begin{figure}[h]
\centering
\includegraphics{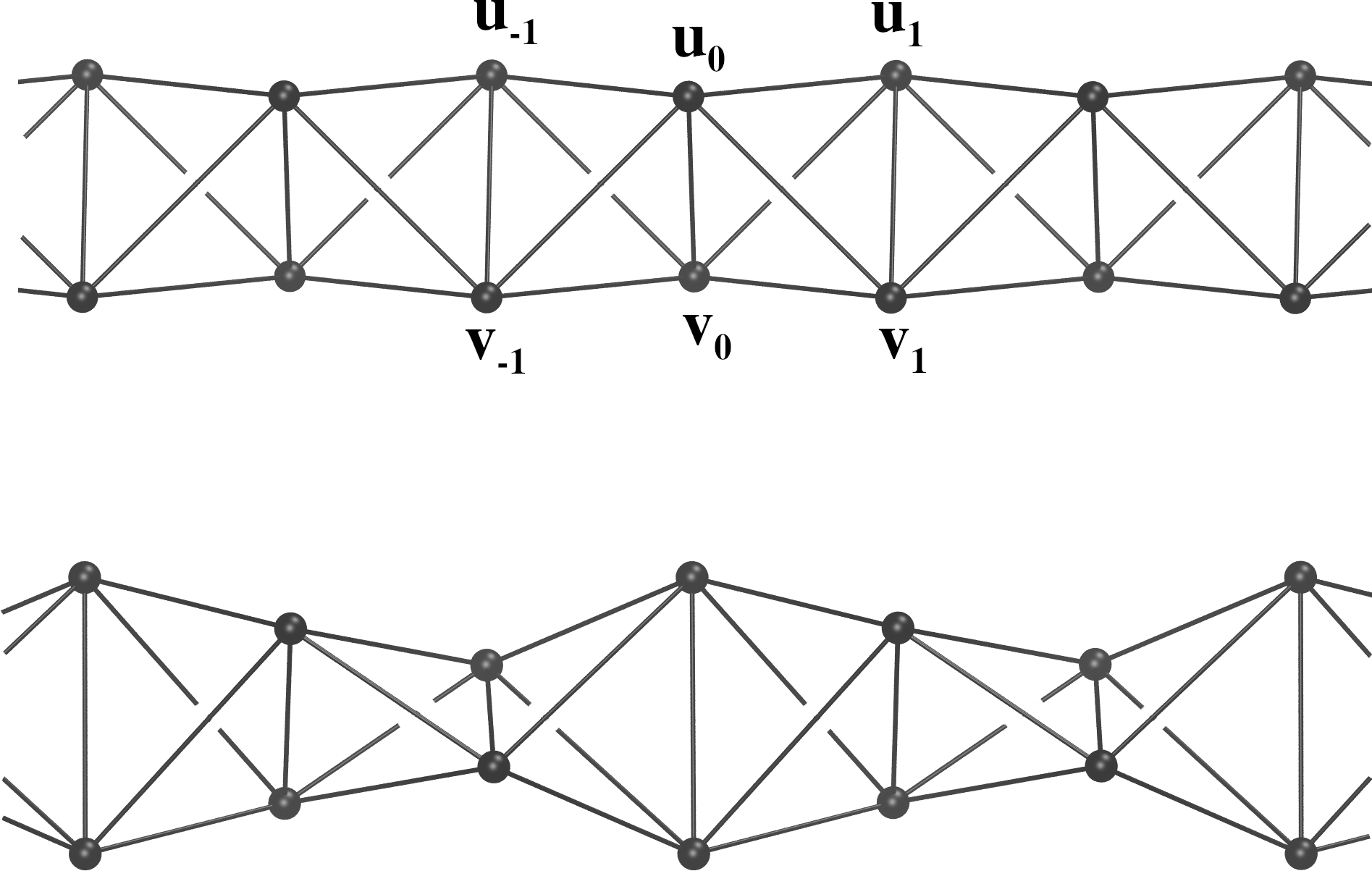}
\caption{Embeddings of the $\bf 3^6(0,2)$ graph with rod-group symmetry $p4_122$ (top) and $p6_122$ (bottom).}
\label{f:infstab}
\end{figure}
\vspace{-2em}
\end{center}

\noindent The generators $\{a, b, ab, ba\}$ can be realized, for example, by the following isometries (\emph{cf.} Table~1):

\begin{center}
\begin{tabular}{l}
$p2/c  \, \, =  \langle \bar1(000), c(x0z), 2(0y\frac{1}{4}), 2(0y\frac{\bar1}{4}) \rangle$;\\
$p4_122 = \langle 2(x00), 4_1(00z), 2(xx\frac{1}{8}), 2(x\bar{x}\frac{\bar1}{8}) \rangle$;\\
$p6_122 = \langle 2(x00), 6_1(00z), 2(xx\frac{1}{12}), 2(x\bar{x}\frac{\bar1}{12}) \rangle$.
\end{tabular}
\end{center}

\begin{center}
\begin{figure}[t]
\centering
\includegraphics{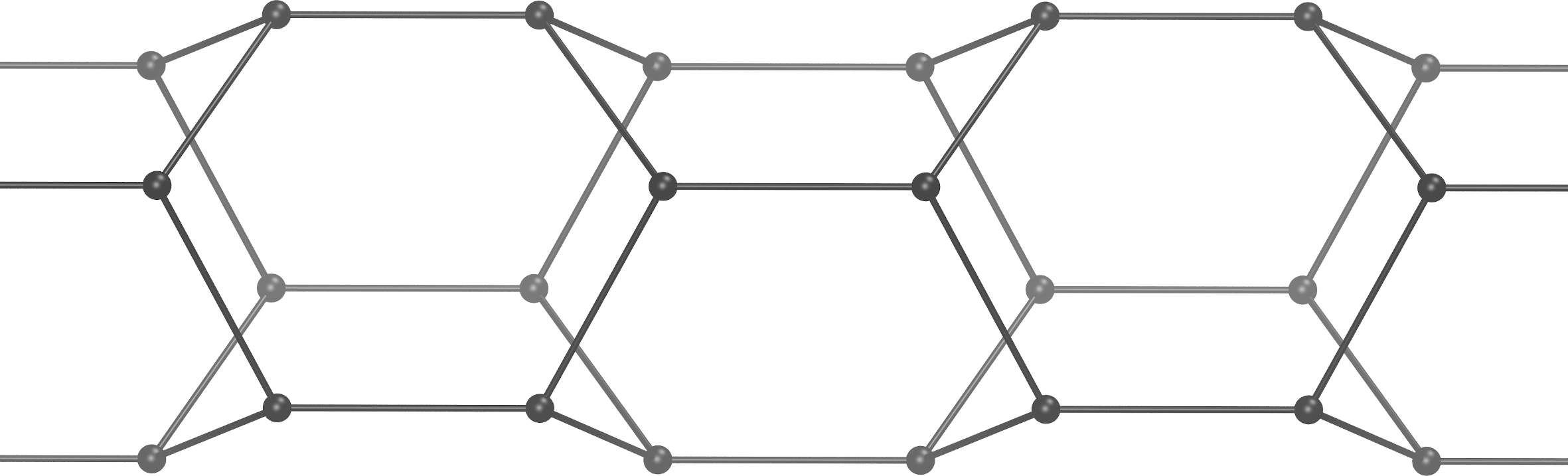}
\caption{The $(3, 0)$ carbon nanotube.}
\label{f:zig}
\end{figure}
\vspace{-2em}
\end{center}

\noindent The orbits of $\langle a \rangle $ on $V(\Gamma)$ are two-element blocks (corresponding to shared edges of tetrahedra) which are the orbits of $B_0(\Gamma)$. The generating sets for $p2/c$ and $p4_122$ give rise to sphere-packing graphs and are listed in \cite{KF78}; the one for $p6_122$  does not allow for such a realization (Fig.~\ref{f:infstab}). Moreover, this graph also admits a vertex-transitive group $G < B(\Gamma)$:

\begin{center}
\begin{tabular}{l}
$C_{\infty} \times C_2 \cong pcc2 = \langle c(x0z), c(0yz), 2(00z) \rangle$.
\end{tabular}
\end{center}

\noindent Here, since $G = pcc2$ is abelian (\emph{cf.} Theorem~5.14 in \cite{Jung84}), the generating set is trivially stabilized by the whole group. $G_0 = \langle 2(00z) \rangle$ generates two-element blocks.

1-Periodic graphs show clearly certain phenomena which are hard to imagine otherwise. For example, \emph{zigzag} carbon nanotubes are graphs where $Aut(\Gamma)_v < B_0(\Gamma)$, \emph{i.e.} vertices are stabilized only by bounded automorphisms of finite order (\emph{e.g.} by a mirror plane of any boat-shaped 6-cycle in Fig.~\ref{f:zig}). For the $(3, 0)$ tube we have: 

\begin{center}
\begin{tabular}{l}
$Aut(\Gamma) = p6_3/mmc, B(\Gamma) = p6_3mc, B_0(\Gamma) = C_{3v}$,\\
$Aut(\Gamma)_v = \langle m(2x,x,z) \rangle$.
\end{tabular}
\end{center}


\medskip

\noindent \emph{2-Periodic graphs}. Simple examples of graphs with bounded automorphisms of finite order are provided by so-called \emph{ladder graphs} which are obtained by linking directly two copies of the same graph. Construct a ladder graph $\Gamma$ from the two square lattices (Fig.~\ref{f:sql}). The automorphism group, $Aut(\Gamma)$, is a layer group:

\begin{center}
\begin{tabular}{l}
$Aut(\Gamma) = p4/mmm, B(\Gamma) = p11m \cong \Z^2 \times C_2, B_0(\Gamma) \cong C_2$,\\
$Aut(\Gamma)_v = \langle m_{xz}, m_{xxz} \rangle = C_{4v} \cong D_4$.
\end{tabular}
\end{center}

\begin{center}
\begin{figure}[h]
\centering
\includegraphics{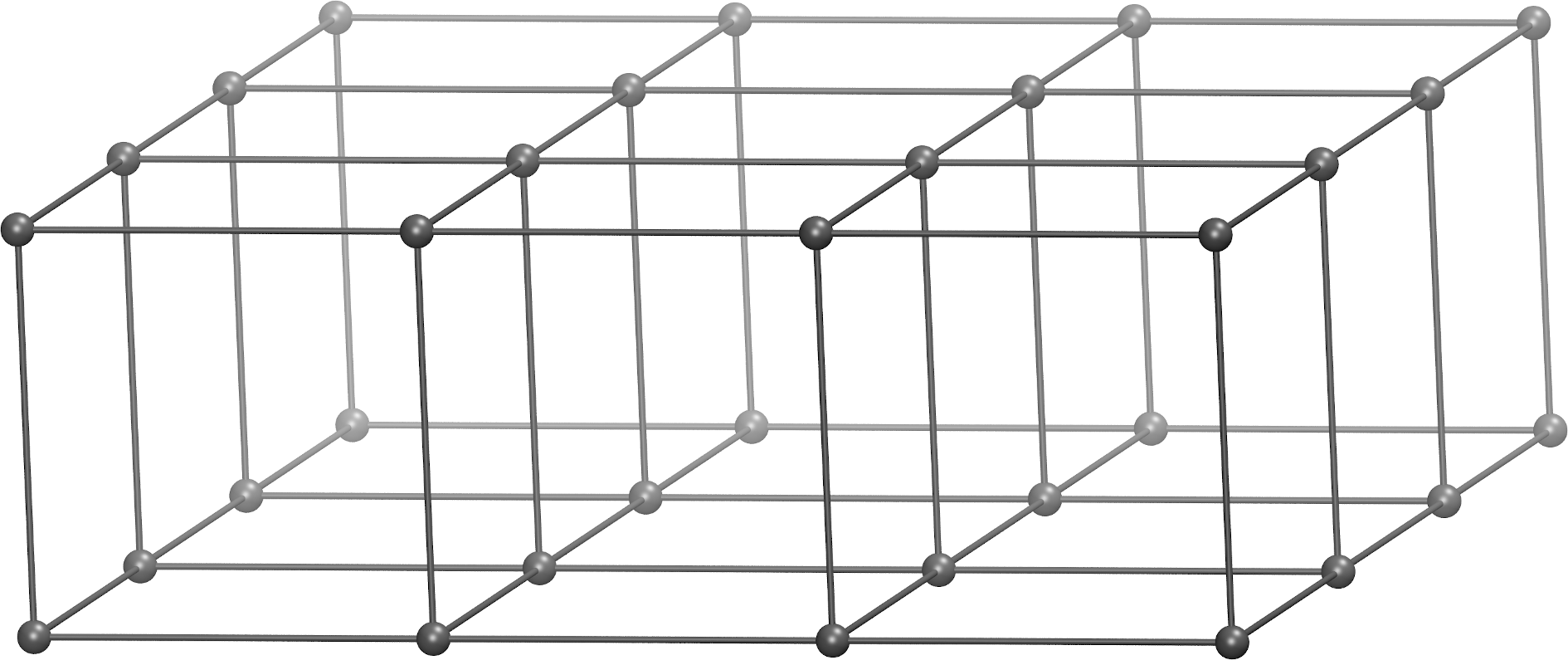}
\caption{A finite portion of the square ladder.}
\label{f:sql}
\end{figure}
\vspace{-1em}
\end{center}

\noindent Owing to the relatively rich vertex stabilizer, this graph has 56 regular groups of automorphisms (up to conjugacy in $Aut(\Gamma)$), 40 of which are isomorphic to the wallpaper groups. As an example, below are given the generating sets for three regular groups which are all isomorphic to the wallpaper group $p4mm$ (\emph{cf.} Table~\ref{t:table2}):

\begin{center}
\begin{tabular}{l}
$p422 = \langle 2(xx0), 4(00z), 4(\frac{1}{2}\frac{1}{2}z) \rangle$;\\
$p\bar4m2 =  \langle 2(xx0), m(x0z), m(0yz), m(x\frac{1}{2}z), m(\frac{1}{2}yz) \rangle$;\\
$p4/nbm = \langle 2(x00), 4(00z), m(x, \frac{1}{2}-x, z), m(\frac{1}{2}+x, x, z) \rangle$.
\end{tabular}
\end{center}

\noindent Clearly each set is stabilized by the respective in-plane 2-fold rotation exchanging the two subgraphs. Again, the orbits of these stabilizers are two-element blocks which are in turn the orbits of $B_0(\Gamma)$.

\medskip

\noindent \emph{3-Periodic graphs}. Following Corollary~\ref{cor3}, Cayley graphs of space groups in any dimension can be constructed so that $B_0(\Gamma) \neq \{1\}$. Such graphs are rare in crystallography mostly because of their undesirable geometrical properties that is, however, not always the case (see next section). Let us generalize the example of $K_{3,3}$ (\emph{cf.} the beginning of Section~\ref{auts}) to space group $G = P321$:

\begin{center}
\begin{tabular}{l}
$S = \langle 2(x00), 2(0y0), 2(xx0), t_x, t_y, t_xt_y, t_z \rangle$;\\ 
$Stab_G(S) \cong D_3$.
\end{tabular}
\end{center}

\noindent The corresponding Cayley graph has valency 11. All its vertices collide in the barycentric placement to the nodes of the hexagonal Bravais lattice. The computation of vertex stabilizer \cite{Bab20} shows that $Aut(\Gamma)_v \cong (D_6 \times C_2) \times D_6$ and $Aut(\Gamma)_v \cap B_0(\Gamma) \cong D_6$.

Another generating set for $P321$ can be constructed by extending that of the $\bf 4^4(0,3)$ tube (Section~\ref{rod1}):

\begin{center}
\begin{tabular}{l}
$S = \langle 3(00z), 2(x00), 2(0y\frac{1}{2}), t_x \rangle$;\\
$Stab_G(S) = \{1\}$;\\
$Stab_{N_{A}(G)}(S) = \langle 2(00z) \rangle$.
\end{tabular}
\end{center}

\noindent Here the computation of vertex stabilizer yields the same result as the conjugation action of the affine normalizer $N_A(P321) = P6/mmm \, ({\bf{a}}, {\bf{b}}, \frac{1}{2}{\bf{c}})$. Therefore, by adding the vertex stabilizer to $S$ we get: $Aut(\Gamma) \cong P622$, though the graph displays collisions in the barycentric placement.

\subsubsection{3-Periodic sphere-packing graphs}

Motivated by the above examples, we analyzed all 1080 three-periodic vertex-transitive sphere-packing graphs published by W.~Fischer and collaborators with respect to bounded automorphisms of finite order. To this end, using the computer algebra system $\GAP$, a code was written which -- given a (finite) subset $S$ of a crystallographic group $G$ -- computes its stabilizer in $G$, $Stab_G(S)$, from a sufficiently large finite quotient $G/T^p \, (p \in \N)$ making use of the conjugacy separability property of space groups \cite[Section~8]{Seg90}. The code exploits the functionalities of $\GAP$ packages \emph{Cryst} \cite{Cryst2019} and \emph{Polycyclic}~\cite{Polyc2020}. The essential prerequisite for this analysis is that the considered graphs be Cayley graphs, \emph{i.e.} admit regular groups of automorphisms. This holds for all but two sphere-packing graphs ($5/3/c2$ and $8/3/c3$), for which the absence of nontrivial bounded automorphsisms of finite order can be concluded \emph{e.g.} from the collision-free barycentric placement. All the other graphs entered the computation of $Stab_G(S)$. The results are collected in Table~\ref{t:table5}. For the nomenclature of graphs see Appendix~B. The generating sets are those listed by Fischer \cite{Fischer74, Fischer93} and Sowa \cite{Sowa2012}. The structure of $Aut(\Gamma)_v$ and, for infinite stabilizers, of $Aut(\Gamma)_v/B_0(\Gamma)_v$ (marked by an asterisk in Table~\ref{t:table5}) computed following \cite{Bab20} is given in the fourth column. The last column contains the isomorphism type of the quotient graph, $\Gamma/\sigma$, with respect to the partition $\sigma$ of $V(\Gamma)$ into orbits of $Stab_G(S)$. 

{\tiny
\begin{table}[t]
\begin{center}
\caption{Sphere-packing graphs with bounded automorphisms of finite order}
\label{t:table5}
\begin{tabular}{|l|l|c|c|c|}
\hline
Name & Generating set $S$ & $Stab_G(S \cup S^{-1})$ & $Aut(\Gamma)_v$ & $\Gamma/\sigma$ \\
\hline
\multicolumn{5}{|c|}{$I432$} \\
\hline
\rule{0pt}{2.5ex} $3/4/c3$  & $4(x00), 2(\frac{1}{2}-x, \frac{1}{4}, x)$ & $\langle 2(\frac{1}{2}y0) \rangle$ & $C_2 \times C_2$ & $3/4/c1$ \\
\rule{0pt}{2.5ex} $4/4/c13$ & $4(x00), 2(\frac{1}{2}+x, \frac{1}{4}, x), 2(\frac{1}{2}-x, \frac{1}{4}, x)$ & $''$ & $C_2 \times C_2$ (*) & $''$ \\
\rule{0pt}{2.5ex} $4/4/c15$ & $4(x00), 2(\frac{1}{2}-x, \frac{1}{4}, x), 2(\frac{1}{2}y0)$  & $''$ & $C_2 \times C_2$ & $''$ \\
\rule{0pt}{2.5ex} $5/3/c21$ & $4(x00), 2(\frac{1}{2}+x, \frac{1}{4}, x), 2(\frac{1}{2}y0),$ & $''$ & $C_2 \times C_2$ (*) &  $''$ \\
& $2(\frac{1}{2}-x, \frac{1}{4}, x)$ & & & \\
\hline
\multicolumn{5}{|c|}{$I4_132$} \\
\hline
\rule{0pt}{2.5ex} $4/3/c22$          & $2(\frac{1}{4}y0), 3(xxx), 2(\frac{1}{4}+x, \frac{1}{8}, x)$ & $\langle 2(\frac{1}{4}-x, \frac{1}{8}, x) \rangle$ & $C_2$ (*) & $3/3/c1$ \\
\rule{0pt}{2.5ex} $\bolit {4/3/c25}$ & $3(xxx), 2(\frac{1}{4}+x, \frac{1}{8}, x), 2(\frac{1}{4}-x, \frac{1}{8}, x)$ & $''$ & $C_2$ & $''$ \\
\rule{0pt}{2.5ex} $\bolit {4/3/c26}$ & $2(\frac{1}{4}y0), 3(xxx), 2(\frac{1}{4}-x, \frac{1}{8}, x)$  & $''$ & $C_2$ & $''$ \\
\rule{0pt}{2.5ex} $5/3/c29$          & $2(\frac{1}{4}y0), 3(xxx), 2(\frac{1}{4}+x, \frac{1}{8}, x),$ & $''$ & $C_2$ (*) & $''$ \\
\rule{0pt}{2.5ex} & $2(\frac{1}{4}-x, \frac{1}{8}, x)$ & & & \\
\hline
\multicolumn{5}{|c|}{$I4_122$} \\
\hline
\rule{0pt}{2.5ex} $\bolit {4/4/t43}$ & $4_3(\frac{1}{4}\frac{1}{4}z), 2(00z), 2(xx0)$ & $\langle 2(xx0) \rangle$ & $D_3$ & $3/10/c1$ \\
\rule{0pt}{2.5ex} $\bolit {5/4/t45}$ & $4_3(\frac{1}{4}\frac{1}{4}z), 4_1(\frac{1}{4}\frac{\bar1}{4}z), 2(\frac{1}{4}y\frac{3}{8})$ & $\langle 2(\frac{1}{4}y\frac{3}{8}) \rangle$ & $C_2 \times C_2$ & $4/4/t1$ \\
\hline
\multicolumn{5}{|c|}{$I4_1/amd$} \\
\hline
\rule{0pt}{2.5ex} $\bolit {4/4/t53}$ & $4_3(\frac{1}{4}\frac{1}{4}z), m(x0z), 2(\frac{1}{4}y\frac{3}{8})$ & $\langle 2(\frac{1}{4}y\frac{3}{8}) \rangle$ & $C_2$ & $3/8/t1$ \\
\hline
\multicolumn{5}{|c|}{$I\bar42d$} \\
\hline
\rule{0pt}{2.5ex} $\bolit {5/4/t45}$ & $d(x,\frac{1}{4}-x,z), d(x,x-\frac{1}{4},z), 2(\frac{1}{4}y\frac{3}{8})$ & $\langle 2(\frac{1}{4}y\frac{3}{8}) \rangle$ & $C_2 \times C_2$ & $4/4/t1$ \\
\hline
\multicolumn{5}{|c|}{$I4_1/a$} \\
\hline
\rule{0pt}{2.5ex} $4/4/t39$ &  $\bar4_z(000), b(xy\frac{1}{8})$ & $\langle 2(00z) \rangle$ & $C_2 \times C_2$ (*) & $3/10/t4$ \\
\rule{0pt}{2.5ex} $5/3/t24$ &  $\bar4_z(000), b(xy\frac{1}{8}), 2(00z)$ & $\langle 2(00z) \rangle$ & $C_2 \times C_2$ (*) & $3/10/t4$ \\
\rule{0pt}{2.5ex} $\bolit {5/4/t45}$ &  $4_3(\frac{1}{4}\frac{1}{4}z), 4_1(\frac{1}{4}\frac{\bar1}{4}z), \bar1(\frac{1}{4}0\frac{3}{8})$ & $\langle \bar1(\frac{1}{4}0\frac{3}{8}) \rangle$ & $C_2 \times C_2$ & $4/4/t1$ \\
\hline
\multicolumn{5}{|c|}{$Fddd$} \\
\hline
\rule{0pt}{2.5ex} $\bolit {4/4/o18}$ &  $2(0y\frac{1}{2}), d(xy\frac{3}{8}), 2(00z)$ & $\langle 2(00z) \rangle$ & $C_2 \times C_2$ &  $3/10/t4$ \\
\rule{0pt}{2.5ex} $4/4/o19$ &  $2(\frac{1}{4}y\frac{1}{4}), 2(x\frac{1}{4}\frac{1}{4}), d(xy\frac{3}{8})$ & $\langle 2(\frac{1}{4}\frac{1}{4}z) \rangle$ & $C_2 \times C_2$ (*) & $3/10/t4$ \\
\rule{0pt}{2.5ex} $5/3/o6$  &  $2(\frac{1}{4}y\frac{1}{4}), 2(x\frac{1}{4}\frac{1}{4}), d(xy\frac{3}{8}), 2(\frac{1}{4}\frac{1}{4}z)$ & $\langle 2(\frac{1}{4}\frac{1}{4}z) \rangle$ & $C_2 \times C_2$ (*) & $3/10/t4$ \\
\hline
\end{tabular}
\end{center}
\end{table}
}


Out of 1080 sphere-packing graphs, there are 16 graphs\footnote{Note that graphs $\{4/3/c25, 4/3/c26\}$, $\{4/4/t39, 4/4/o19\}$, $\{5/3/t24, 5/3/o6\}$ are pairwise isomorphic.} which possess nontrivial bounded automorphisms of finite order. Six of them are ladder graphs (highlighted in bold in Table~\ref{t:table5}), and others are related to them by addition or removal of some edges. In all the cases the orbits of $B_0(\Gamma)$ are two-element blocks, which are `reconstructed' from the orbits of $Stab_G(S)$ on~$V(\Gamma)$ (Corollary~\ref{cor2}). It is noteworthy that almost half of the graphs have uncountable automorphism groups. Note that our results agree in part with those of Eon (2013) \cite{Eon2013} who considered $4/3/c25\text{-}26$, $4/4/o18\text{-}19$ and $5/3/o6/t24$ using the formalism of labelled quotient graphs.

Among the considered graphs, the sphere packing $3/8/t6$ (not in Table~\ref{t:table5}) deserves some comments. Fischer (1993) found it only in the general position of space group $I\bar42d$. The set of generators can be chosen as $\{2(x\frac{1}{4}\frac{1}{8}), d(x, \frac{1}{4}-x, z)\}$, the stabilizer of which is trivial. The graph has collisions in the barycentric placement that is, however, neighbour-unique. This allows us to get $Aut(\Gamma) = I4_1/amd$. The same result is obtained if the stabilizer of the generating set is referred to the affine normalizer of $I\bar42d$. The vertex stabilizer has order 2, and is unexpectedly (for valency 3) generated by the inversion through $(0\frac{1}{4}\frac{1}{8})$ forcing vertices to collide in $\R^3$ to the invariant position $I4_1/amd \, 8e \, (.2/m.)$. Therefore, regular groups of $3/8/t6$ must be noncentrosymmetric ($I\bar42d, I4_122, I4_1md, I\bar4m2$). Edge intersections are retained in the groups with mirror planes. Space groups $I\bar42d$ and $I4_122$ allow for piecewise-linear embeddings of the graph in $\R^3$. The first variant is realised with shortest equal distances and tabulated by Fischer (1993), the second [generators: $2(x\frac{1}{4}\frac{1}{8}), 2(xx0), 2(x,\frac{1}{2}+x, \frac{1}{4})$] does not correspond to a sphere packing. Both embeddings are compared in Fig.~\ref{f:beide}.

The pair of sphere packings $4/3/c25\text{-}26$ is remarkable as they arise from non-isotopic embeddings of one and the same ladder graph built from the two copies of $3/3/c1$. This raises a natural question -- given a three-periodic vertex-transitive ladder graph, how to find its (piecewise linear) vertex-transitive embeddings in $\R^3$? A practical algorithm for this problem is suggested in the next section.

\begin{center}
\begin{figure}[ht]
\centering
\includegraphics{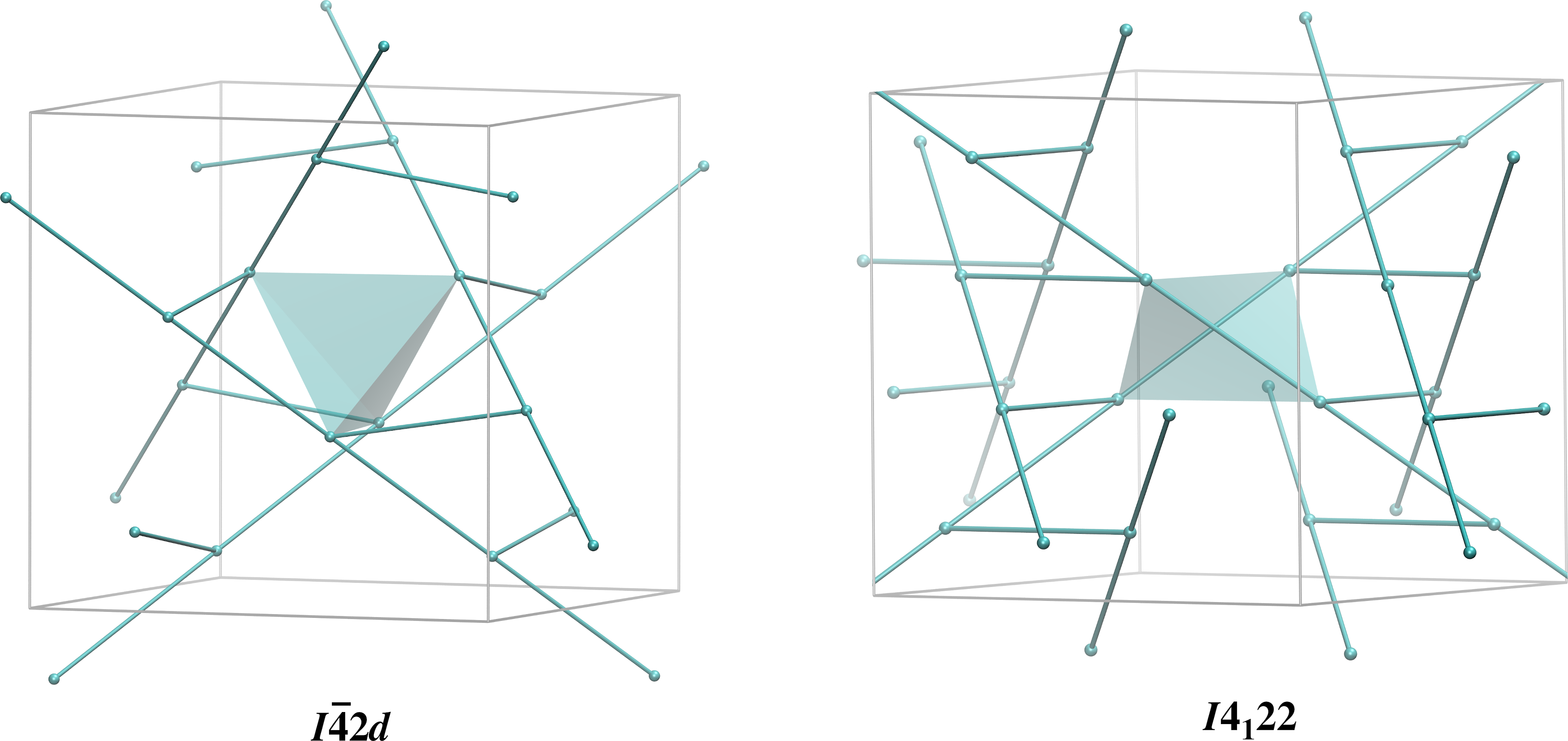}
\caption{Embeddings of the $3/8/t6$ sphere-packing graph with different symmetry. The tetrahedral void at the center of the unit cell is emphasized in both structures. Notice also different point-group symmetry of the tetrahedra ($\bar4$ \emph{versus} $222$).}
\label{f:beide}
\end{figure}
\end{center}

\subsubsection{Construction of 3-periodic \emph{ladder graphs} from $\R^4$}

A suitable (piecewise-linear) embedding of a periodic graph in $\R^d$ may be obtained in two steps: (a) compute a symmetry group of the embedding based on the knowledge of $Aut(\Gamma)$, (b) given a symmetry group, find coordinates which are compatible with it and are optimal in some sense. The first step is group-theoretic, the second is numerical. It is the first step that we address here for 3-periodic vertex-transitive ladder graphs in $\R^3$ (this topic has been considered recently by O'Keeffe \emph{et al.} \cite{OK25} in an essentially empirical way). The second step is carried out by setting up a system of quadratic equations for the distances to be solved usually by least-squares methods.

We start with a 3-periodic vertex-transitive graph $\Gamma$ with nontrivial vertex stabilizers in $Aut(\Gamma)$ that is isomorphic to a space group, lift it to a 3-dimensional hyperplane of $\R^4$, say, $w = 1/5$, and the reflection in the hyperplane $w = 0$ generates a second copy of it. In this way a 3-periodic ladder graph $\Lambda$ is obtained for which $Aut(\Lambda) \cong Aut(\Gamma) \times C_2$, and vertex stabilizers are isomorphic to those of the parent graph $\Gamma$ in $Aut(\Gamma)$. Now, compute proper vertex-transitive subgroups of $Aut(\Lambda)$ as subperiodic groups in $\R^4$ (with a three-dimensional translation lattice), and project them back to $\R^3$ by restricting symmetry operators to the respective subspace. The groups not containing the reflection at $w=0$ (if they exist) are precisely space groups for vertex-transitive embeddings of~$\Lambda$ in $\R^3$. The computed generating sets for vertex-transitive groups consist of elements mapping an arbitrary vertex to the adjacent vertices. These elements are then used to set up a system of quadratic equations for relaxing vertex positions keeping the symmetry.

\medskip

\noindent \emph{Example}. Construct a ladder graph from $\Gamma = 3/3/c1$ using the proposed approach. We have:

\begin{center}
\begin{tabular}{l}
\rule{0pt}{2.5ex} $Aut(\Gamma) = I4_132, Aut(\Gamma)_v = \langle 2(\frac{1}{4}-x, \frac{1}{8}, x) \rangle$;\\
\rule{0pt}{2.5ex} $Aut(\Lambda) \cong I4_132 \times C_2, Aut(\Lambda)_v \cong C_2$.
\end{tabular}
\end{center}

\noindent It is easy to see that proper vertex-transitive subgroups of $Aut(\Lambda)$ -- without the mirror at $w=0$ -- are isomorphic to $I4_132$ and, being `four-dimensional layer groups', are in a one-to-one correspondence with those index-2 subgroups of $I4_132$ which act transitively on $V(\Gamma)$. Out of three index-2 subgroups of $I4_132$, two ($I2_13$ and $P4_132$) are transitive on $V(\Gamma)$ and one ($P4_332$) is not (since $Aut(\Gamma)_v < P4_332)$. In the actual computation, $Aut(\Lambda)$ and the two vertex-transitive subgroups, $G_1$ and $G_2$, arise as quotients of four-dimensional space groups $25/10/2/3$, $24/4/3/3$ and $25/10/1/8$ (in the notation of \cite{BBNWZ}), respectively, obtained by factoring out the normal translation subgroup of rank~1. The point groups are $P(G_1) \cong Sym(4)$ and $P(Aut(\Lambda)) = P(G_2) \cong Sym(4) \times C_2$, hence $G_1$ and $G_2$ are \emph{translationengleiche} and \emph{klassengleiche} subgroups of $Aut(\Lambda)$, resp. Upon restricting symmetry operators to $\R^3$, the glide reflection of $G_2$, $\{m_{xyz0}|\frac{1}{2}\frac{1}{2}\frac{1}{2}0\}$, is mapped to the $I$-translation of $I4_132$ with a simultaneous reduction of the point group, and finally we get $G_1 \cong G_2 \cong I4_132$. Note that $G_1$ (related to the  pair $I4_132-I2_13$) corresponds to $4/3/c26$ and $G_2$ (related to the pair $I4_132-P4_132$) to $4/3/c25$. Both graphs are illustrated in Fig.~\ref{f:ulds}.

\vspace{-1em}

\begin{center}
\begin{figure}[h]
\centering
\includegraphics{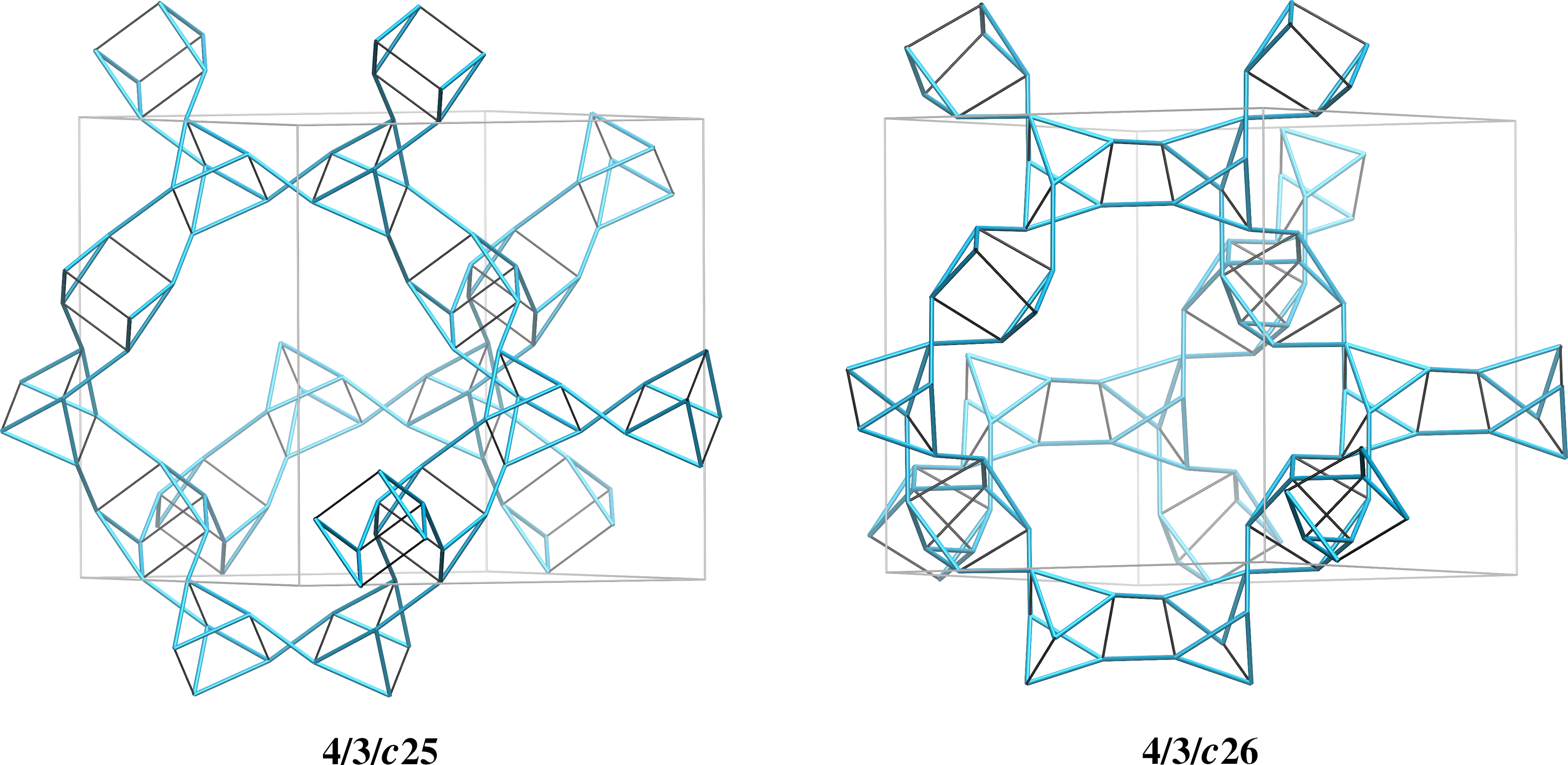}
\caption{$4/3/c25$ and $4/3/c26$ sphere-packing graphs. Two connected components are joined by thin edges in both graphs.}
\label{f:ulds}
\end{figure}
\end{center}
\vspace{-2em}

{\tiny
\begin{table}[H]
\begin{center}
\caption{Examples of ladder-graph construction}
\label{t:table6}
\begin{tabular}{|c|l|l|l|l|}
\hline
$\Gamma$ & $G(\Lambda)$ & $G(\Lambda)_{\{\Gamma\}}$ & Generators for $G(\Lambda)$ & Source/ \\
& & & & related graph \\
\hline
\rule{0pt}{2.5ex} $3/4/c1$ & $Pn\bar3m$ & $Fd\bar3c \, (2{\bf{a}})$ & $\bar4_z(00\frac{1}{2}), 2(\frac{1}{4},y,\frac{1}{2}+y), 2(x0\frac{1}{2})$ & see text \\
         &            & $Fd\bar3m \, (2{\bf{a}})$ & $m(xxz), m(x\bar{x}z), 2(\frac{1}{4}, y, \frac{\bar1}{2}-y), 2(x0\frac{\bar1}{2})$ & \cite{Blatov07} \\
$3/8/h1$ & $R\bar3m$  & $R3m$ & $3_2(\frac{1}{3}0z), m(x, 2x-1, z), 2(x00)$ & -- \\
         &            & $R\bar3m \, (2{\bf{c}})$ & $2(x,x-\frac{1}{3},\frac{1}{3}), 2(\frac{1}{3}y\frac{2}{3}), m(x, 2x-1, z), 2(x0\frac{1}{2})$ & \cite{Blatov07} \\
$3/8/t1$ & $I4_1/amd$ & $I4_1md$ & $4_3(\frac{1}{4}\frac{1}{4}z), m(0yz), 2(x\frac{1}{4}\frac{1}{8})$ & $4/4/t53$ \\
         &            & $I\bar4m2$ & $2(xx0), m(0yz), 2(x, \frac{1}{2}-x, \frac{1}{4}), 2(x\frac{1}{4}\frac{1}{8})$ & -- \\
\rule{0pt}{2.5ex} $3/10/c1$ & $I4_122$  & $I2_12_12_1$ & $2(x\frac{1}{4}\frac{\bar3}{8}), 2(\frac{1}{4}y\frac{\bar1}{8}), 2(\frac{1}{2}0z), 2(x, \frac{1}{2}-x, \frac{\bar1}{4})$ & $4/3/c26$ \\
\rule{0pt}{2.5ex}  &  & $I4_1$ & $4_3(\frac{1}{4}\frac{1}{4}z), 2(\frac{1}{2}0z), 2(x, \frac{1}{2}-x, \frac{\bar1}{4})$ & $4/4/t43$ \\
          &           & $P4_32_12$ & $4_3(\frac{1}{4}\frac{1}{4}z), 2(x, x-\frac{1}{2}, \frac{\bar1}{4}), 2(x, \frac{1}{2}-x, \frac{\bar1}{4})$ & $4/3/c25$ \\
\rule{0pt}{2.5ex} &   & $P4_122$ & $2(x\frac{1}{4}\frac{\bar3}{8}), 2(\frac{1}{4}y\frac{\bar1}{8}), 2(x, x-\frac{1}{2}, \frac{\bar1}{4}), 2(x, \frac{1}{2}-x, \frac{\bar1}{4})$ & -- \\
\rule{0pt}{2.5ex} $3/12/h1$ & $P6_222$  & $P6_2$   & $6_2(00z), 2(\frac{1}{2}0z), 2(x00)$ & -- \\
          &           & $P3_212$ & $2(x\bar{x}\frac{\bar1}{6}), 2(2x,x,\frac{1}{6}), 2(x, 2x-1, 0), 2(x00)$ & -- \\
          &           & $P6_422 \, (2{\bf{c}})$ & $2(\frac{1}{2}0z), 2(x\bar{x}\frac{1}{3}), 2(2x, x, \frac{2}{3}), 2(x0\frac{1}{2})$ & -- \\
          &           & $P6_122 \, (2{\bf{c}})$ & $6_2(00z), 2(x, 2x-1, \frac{1}{2}), 2(x0\frac{1}{2})$ & -- \\
\hline
\end{tabular}
\end{center}
\end{table}
}

To demonstrate the capability of the proposed approach, in Table~\ref{t:table6} regular groups of automorphisms $G(\Lambda)$ are given for ladder graphs built from selected graphs of different symmetry (only regular groups are listed in order to save space, and moreover, they are helpful for understanding the method). Therein, the last element of a generating set in the fourth column is always one exchanging the two 3-periodic subgraphs isomorphic to $\Gamma$, \emph{i.e.} it is the element from the nontrivial coset of $G/G_{\{\Gamma\}}$. This coset must not contain mirror reflections because otherwise edge intersections are enforced in the embedding \cite{Bab16}. For example, with symmetry $Pn\bar3m$ the ladder graph can be constructed from $3/4/c1$ in two different ways. Consider the variant corresponding to the first entry of Table~\ref{t:table6}. In this case $G_{\{\Gamma\}} = Fd\bar3c$, \emph{i.e.} the group without diagonal mirrors of $G=Pn\bar3m$. As a result, this generating set always leads to edge intersections in the embedding of $\Lambda$.

Table~\ref{t:table6} shows that known ladder graphs are reproduced, and more are found. For example, a new embedding is derived for the ladder graph of $3/10/c1$ (see coordinates in the Supporting information). The ladder graphs for $3/4/c1$ and $3/8/h1$ from Blatov (2007) \cite{Blatov07} are reproduced which (in his nomenclature) should be named {\bf nbo-g}-4-$Pn\bar3m$ and {\bf smd}-4-$R\bar3m$, resp. Our method also allows to clarify why certain graphs could not be encountered in former studies: for example, ladder graphs built from vertex-transitive graphs with vertex stabilizers of order~3 (\emph{e.g.} $4/6/c4$) do not admit any vertex-transitive space group; those built from vertex-transitive graphs with vertex stabilizers of order~2 generated by mirror reflections ($3/6/h1$~\emph{etc.}) show up edge intersections in any vertex-transitive embedding because mirrors necessarily belong to $G \setminus G_{\{\Gamma\}}$.

\section{Appendix A: More group theory} 

There is a deep connection between the structure of a vertex-transitive graph and the subgroup structure of its automorphism group. The asymptotic behaviour of the growth function of a graph (coordination sequence in crystallography) is the property that determines this connection. The problem can be put into a group-theoretic setting as described below (it is enough to keep in mind here that the growth of a group is equivalent to that of its Cayley graph).

\begin{theorem}
\label{teo3}
\emph{(Gromov, 1981)}. If $G$ is a finitely generated group with polynomial growth, then it contains a nilpotent subgroup of finite index.
\end{theorem}

\noindent This means that if $\Gamma$ is a Cayley graph of $G$, then a finite-index nilpotent subgroup of $G$ acts on $V(\Gamma)$ with a finite number of orbits. Note that the class of torsion-free nilpotent groups of growth degree $d$ coincides with that of free abelian groups $\Z^d$ for $d \leq 3$ \cite{Seif97}. However, not all vertex-transitive graphs are Cayley graphs, \emph{i.e.} not all allow for a regular group of automorphisms. The following theorem covers a significantly more general case including graphs with uncountable automorphism groups (therein, $x^{\sigma}$ is the block of the partition $\sigma$ containing vertex $x$).

\begin{theorem}
\label{teo4}
\emph{(Trofimov, 1984)}. Let $\Gamma$ be a connected locally finite vertex-transitive graph with polynomial growth. There exists an $Aut(\Gamma)$-invariant partition $\sigma$ of $V(\Gamma)$ into finite blocks $x^\sigma$ such that the group $Aut(\Gamma/\sigma)$ is finitely generated and contains a nilpotent subgroup of finite index. Moreover, the stabilizer $Aut(\Gamma/\sigma)_{x^\sigma}$ is finite.
\end{theorem}

The blocks of $\sigma$ in Theorem~\ref{teo4} can be chosen to be the orbits of $B_0(\Gamma) \vartriangleleft Aut(\Gamma)$ \cite[p.~416]{Trofimov84}. Consider the natural homomorphism $Aut(\Gamma) \rightarrow Aut(\Gamma)^{\sigma}$ with kernel $B_0(\Gamma)$. By Theorem~\ref{teo4}, the partition~$\sigma$ of $V(\Gamma)$ into $B_0(\Gamma)$-orbits defines the graph $\Gamma/\sigma$ with $Aut(\Gamma)^{\sigma}$ as the group of automorphisms preserving~$\sigma$ (clearly $Aut(\Gamma)^{\sigma} \leq Aut(\Gamma/\sigma)$). For vertex stabilizers the following useful relation holds: $Aut(\Gamma)^{\sigma}_{x^{\sigma}} \cong Aut(\Gamma)_x/B_0(\Gamma)_x$ \cite[Proposition~1.4(b)]{Trofimov86}. For a Cayley graph $\Gamma$ of a crystallographic space group, this means (using the arguments from Section~4.1, Proposition~\ref{prop6}) that, if $B_0(\Gamma) \neq \{1\}$, then $Aut(\Gamma)_x$ does not consist of bounded automorphisms alone.

Theorem~\ref{teo4} uncovers the \emph{only} `mechanism' for graphs of polynomial growth to have uncountable $Aut(\Gamma)$: it is the possibility to permute vertices inside blocks $x^\sigma$ (such as the blocks $\{u_i, v_i\}$ in Fig.~\ref{f:infstab}) more or less independently. Such automorphisms have finite order and belong to the group~$B_0(\Gamma)$. $B_0(\Gamma)$ is \emph{locally finite} \cite{Seif89all}, that is, the group whose every finitely generated subgroup is finite. This in turn implies that, if $B_0(\Gamma)$ is an infinite group, then $B_0(\Gamma)$ and $B_0(\Gamma)_x$ are \emph{not finitely generated}.

Theorems~\ref{teo3} and \ref{teo4} point towards the emergence of translational symmetry in crystal structures. The following result deserves to be well-known in solid-state sciences.

\begin{theorem}
\label{teo5}
\emph{(Seifter and Trofimov, 1997)}. Let $\Gamma$ be a connected locally finite graph with polynomial growth of degree $d \leq 2$, and let $Aut(\Gamma)$ have finitely many orbits on $V(\Gamma)$. Then there exists a subgroup $H \leq Aut(\Gamma)$ that is isomorphic to $\Z^d$ and acts on $V(\Gamma)$ with finitely many orbits.
\end{theorem}

\noindent Theorem~\ref{teo5} asserts the emergence of translational symmetry from the growth rate \emph{and} the action of $Aut(\Gamma)$ on $V(\Gamma)$ with finitely many orbits -- both conditions have a natural physical meaning. Note that the index $|Aut(\Gamma) : H|$ can be infinite (that implies infinite vertex stabilizers in $Aut(\Gamma)$), hence the beauty and significance of this theorem. For example, the famous Penrose pattern has surely quadratic growth but lacks translational periodicity because the number of orbits of $Aut(\Gamma)$ on $V(\Gamma)$ is infinite.

\section{Appendix B: Nomenclature for periodic graphs}
The nomenclature for periodic graphs proposed in \cite{Fischer71, KF78} is preferred for many reasons (one of them is that most subperiodic graphs are not included, for example, in the popular RCSR database \cite{RCSR}). 1-Periodic graphs are characterized by the Schl{\"a}fli symbol of a vertex-transitive planar graph followed by the rolling vector (referred to the usual choice of basis vectors \cite{ITE}). For example, the famous Boerdijk--Coxeter helix would be called $\bf 3^6(1, 3)$. For three-periodic sphere-packing graphs, the symbols used are of type $k/m/fn$: here $k$ stands for the valency, $m$ for the girth (length of a shortest cycle) of a graph, $f$ refers to the crystal system (where the corresponding sphere packing is realized with highest symmetry), and $n$ is an arbitrary number. Table~\ref{t:table7} contains correspondences (if available) between the original Fischer names and the RCSR three-letter codes.

{\tiny
\begin{table}[H]
\begin{center}
\caption{Fischer symbols vs. RCSR codes}
\label{t:table7}
\begin{tabular}{|c|c|c|c|}
\hline
Fischer & RCSR  & Fischer & RCSR \\
\hline
$3/3/c1$  & {\bf srs-a} & $6/4/c1$  & {\bf pcu} \\
$3/4/c1$  & {\bf nbo-a} & $8/3/c3$  & {\bf pyc} \\
$3/10/c1$ & {\bf srs}   & $3/8/t1$  & {\bf lig} \\
$4/3/c25$ & {\bf uld}   & $3/10/t4$ & {\bf ths} \\
$4/3/c26$ & {\bf uld-z} & $4/4/t1$  & {\bf lvt} \\
$4/4/c15$ & {\bf ulm}   & $4/6/t4$  & {\bf cds} \\
$4/6/c1$  & {\bf dia}   & $3/6/h1$  & {\bf pcu-h} \\
$4/6/c4$  & {\bf gsi}   & $3/8/h1$  & {\bf etb} \\
$5/3/c2$  & {\bf fnn}   & $3/12/h1$ & {\bf twt} \\
$5/3/c21$ & {\bf nbo-g} & & \\
$5/3/c29$ & {\bf fco}   & & \\
\hline
\end{tabular}
\end{center}
\vspace{-1em}
\end{table}
}

\bigskip 

\noindent {\bf Acknowledgements}. The author thanks Prof.~Vladimir Trofimov for many years of helpful correspondence and for critical comments on the manuscript. Part of this work was done in 2022--2023 when the author served as a substitute professor (Vertretungsprofessor) at the Ludwig-Maximilians-Universit\"at M\"unchen, Sektion Kristallographie. Funding within the Hightech Agenda Bayern is gratefully acknowledged.

\medskip

\noindent {\bf Supporting information}. Generating sets and invariants for rod and layer groups, and invariants for three-dimensional space groups as a $\GAP$-readable file, the $\GAP$ script for computing the stabilizer of a generating set for a crystallographic group, coordinates for the embeddings of graphs mentioned in the paper. All the materials are available from the author upon request.

\medskip

\noindent {\bf Comment}. This article has been submitted for publication in \emph{Acta Cryst.~A}.


\newpage

\begin{thebibliography}{99}

\setlength\itemsep{-0.01cm}

\begin{small}

\bibitem{Bab16} Baburin, I.~A. On the group-theoretical approach to the study of interpenetrating nets. Acta Cryst. 2016, A72, 366-375.

\bibitem{Bab20} Baburin, I.~A. On Cayley graphs of $\Z^4$. Acta Cryst. 2020, A76, 584-588.

\bibitem{Bab26} Baburin, I.~A. Short presentations for crystallographic groups. Acta Cryst. 2026, A82, 18-31.

\bibitem{Evar09} Bandura, A.~V. and Evarestov, R.~A. From anatase $(101)$ surface to TiO\textsubscript{2} nanotubes: rolling procedure and ﬁrst-principles LCAO calculations. Surface Science, 2009, 603, L117-L120.

\bibitem{Besche2002} Besche, H.~U., Eick B., O'Brien, E.~A. A millenium project: constructing small groups. Int. J. Algebra Comput., 2002, 12, 623-644.

\bibitem{Biggs93} Biggs, N. Algebraic graph theory, 2nd edition. Cambridge University Press, 1993.

\bibitem{Blatov07} Blatov, V.~A. Topological relations between three-dimensional periodic nets. I. Uninodal nets. Acta Cryst. 2007, A63, 329-343.

\bibitem{BBNWZ} Brown, H., B{\"u}low, R., Neub{\"u}ser, J., Wondratschek, H., Zassenhaus, H. Crystallographic groups of four-dimensional space. Wiley, 1978.

\bibitem{Bur45} Burckhardt, J.~J. Die Bewegungsgruppen der doppelt z{\"a}hlenden Ebene. In: Festschrift zum 60. Geburtstag von Prof.~Dr. Andreas Speiser. Z{\"u}rich: F{\"u}ssli, 1945, pp.~153-159. https://zbmath.org/0060.45903

\bibitem{Proser03} Carlucci, L., Ciani, G. and Proserpio, D.~M. Polycatenation, polythreading and polyknotting in coordination network chemistry. Coord. Chem. Rev., 2003, 246, 247-289.

\bibitem{Linegr10}  Damnjanovi{\'c}, M., Milo{\v{s}}sevi{\'c}, I. Line groups in physics. Springer, 2010.

\bibitem{Delgado2004} Delgado-Friedrichs, O. Barycentric drawings of periodic graphs. Lecture Notes in Computer Science, vol. 2912. Springer, 2004, 178-189.

\bibitem{OK25} Delgado-Friedrichs, O., O'Keeffe, M. and Treacy, M.~M.~J. Periodic graphs with coincident edges: folding-ladder and related graphs. Acta Cryst. 2025, A81, 49-56.

\bibitem{Cryst2019} Eick, B., G{\"a}hler, F., Nickel, W. \emph{Cryst} -- computing with crystallographic groups, a refereed $\GAP$~4 package.

\bibitem{Polyc2020} Eick, B., Nickel, W., Horn, M. \emph{Polycyclic} -- computation with polycyclic groups, a refereed $\GAP$~4 package.

\bibitem{Eon2005} Eon, J.-G. Graph-theoretical characterization of periodicity in crystallographic nets and other infinite graphs. Acta Cryst. 2005, A61, 501-511.

\bibitem{Fischer71} Fischer, W. Existenzbedingungen homogener Kugelpackungen in Raumgruppen tetragonaler Symmetrie. Z. Kristallogr., 1971, 133, 18-42.


\bibitem{Fischer74} Fischer, W. Existenzbedingungen homogener Kugelpackungen zu kubischen Gitterkomplexen mit drei Freiheitsgraden. Z.~Kristallogr., 1974, 140, 50-74.

\bibitem{Fischer93} Fischer, W. Tetragonal sphere packings. III. Lattice complexes with three degrees of freedom. Z.~Kristallogr., 1993, 205, 9-26.

\bibitem{GAP} The $\GAP$ Group. $\GAP$ -- Groups, Algorithms, Programming. Version 4.19. Available from https://www.gap-system.org.

\bibitem{Seif89all} Godsil, C.~D., Imrich, W., Seifter, N., Watkins, M.~E. and Woess, W. A note on bounded automorphisms of infinite graphs.  Graphs and Combinatorics, 1989, 5, 333-338.

\bibitem{Gromov81} Gromov, M. Groups of polynomial growth and expanding maps. Publ. Math. Inst. Hautes \'{E}tudes Sci., 1981, 53, 53-78.

\bibitem{Gromov83} Gromov, M. Infinite groups as geometric objects. ICM Proceedings, Warszawa, 1983, Vol.~1, 385-392.

\bibitem{Seif91} Imrich, W. and Seifter, N. A survey on graphs with polynomial growth. Discrete Math., 1991, 95, 101-117.

\bibitem{ITE} International tables for crystallography. Volume E (2006). Available from https://it.iucr.org.

\bibitem{Jones11} Jones, G.~A. Detecting and identifying wallpaper groups. Symmetry: Culture and Science, 2011, 22, pp. 131-157.

\bibitem{Jung84} Jung, H.~A. and Watkins, M.~E. Fragments and automorphisms of infinite graphs. Eur. J. Comb., 1984, 5, 149-162.

\bibitem{Twist2012} Kato, K., Koretsune, T. and Saito, S. Energetics and electronic properties of twisted single-walled carbon nanotubes. Phys.~Rev.~B 85, 115448.

\bibitem{KF75} Koch, E. and Fischer, W. Automorphismengruppen von Raumgruppen und die Zuordnung von Punktlagen zu Konfigurationslagen. Acta Cryst. 1975, A31, 88-95.

\bibitem{KF78b} Koch, E. and Fischer, W. Complexes for crystallographic point groups, rod groups and layer groups. Z. Kristallogr., 1978, 147, 21-38.

\bibitem{KF78} Koch, E. and Fischer, W. Types of sphere packings for crystallographic point groups, rod groups and layer groups. Z. Kristallogr., 1978, 148, 107-152.

\bibitem{Koe80} K{\"o}hler, K.-J. On the structure and the determination of $n$-dimensional partially periodic crystallographic groups. \emph{MATCH}, 1981, 10, 27-53.

\bibitem{Litv2000} Litvin, D.~B. and Kopsk{\'y}, V. Subperiodic groups isomorphic to factor groups of reducible space groups. Acta Cryst. 2000, A56, 370-374.

\bibitem{layer25} Mahmoudi, S., Dresselhaus, E.~J. and Dimitriyev, M.~S. An orbifold framework for classifying layer groups with an application to knitted fabrics. Interface Focus, 2026, 16, 20250093.

\bibitem{Eon2013} Moreira de Oliveira Jr, M. and Eon, J.-G. Non-crystallographic nets with finite blocks of imprimitivity for bounded automorphisms.  Acta Cryst. 2013, A69, 276-288.

\bibitem{RCSR} O'Keeffe, M., Peskov, M.~A., Ramsden, S.~J. and Yaghi, O.~M. The Reticular Chemistry Structure Resource (RCSR) database of, and symbols for, crystal nets. Acc. Chem. Res. 2008, 41, 1782-1789. http://rcsr.net

\bibitem{OK21} O'Keeffe, M. and Treacy, M.~M.~J. Isogonal piecewise linear embeddings of 1-periodic weaves and some related structures. Acta Cryst. 2021, A77, 130-137.

\bibitem{Rat15} Ratcliffe, J.~G., Tschantz, S.~T. On the isometry group of a compact flat manifold. Geom. Dedicata, 2015, 177, 43-60.

\bibitem{Sabidussi64} Sabidussi, G. Vertex-transitive graphs. Monatshefte f{\"u}r Mathematik, 1964, 68, 426-438.

\bibitem{Salle19} de la Salle, M., Tessera, R. Characterizing a vertex-transitive graph by a large ball. Journal of topology, 2019, 12, 704-742.

\bibitem{Tes21} Tessera, R. and Tointon, M.~C.~H. A finitary structure theorem for vertex-transitive graphs of polynomial growth. Combinatorica, 2021, 41, 263-298.

\bibitem{Seg90} Segal, D. Decidable properties of polycyclic groups. Proc. London Math. Soc., 1990, 61, 497-528.

\bibitem{Seif97} Seifter, N. and Trofimov, V.~I. Automorphism groups of graphs with quadratic growth. J. Comb. Theory B, 1997, 71, 205-210.

\bibitem{Souv26} Souvignier, B. It's all in the group. Acta Cryst., 2026, A82, 79-82.

\bibitem{Sowa2012} Sowa, H. Orthorhombic sphere packings. IV. Trivariant lattice complexes of space groups without mirror  planes belonging to crystal class $mmm$.  Acta Cryst., 2012, A68, 763-777.

\bibitem{Tom84} Tomkinson, M.~J. FC-groups. Res. Notes in Math. 96. Pitman, London, 1984.

\bibitem{Trofimov83} Trofimov, V.~I. Automorphisms of graphs and a characterization of lattices. Math. USSR Izvestiya, 1984, 22, 379-391. 

\bibitem{Trofimov84} Trofimov, V.~I. Graphs with polynomial growth. Math. USSR Sbornik, 1985, 51, 405-417. 

\bibitem{Trofimov85} Trofimov, V.~I. Automorphism groups of graphs as topological groups. Math. Notes USSR, 1985, 38, 717-720.

\bibitem{Trofimov86} Trofimov, V.~I. The action of a group on a graph. Math. USSR Izvestiya, 1987, 29, 429-447.

\bibitem{Trofimov2012} Trofimov, V.~I. A note on the extendability of an isomorphism of subgraphs of a graph to an automorphism of the graph. Trudy Inst. Math. Mekh., 2012, 18, 26-29. https://mi.mathnet.ru/timm835

\bibitem{Trofimov2014} Trofimov, V.~I. The finiteness of the number of symmetrical 2-extensions of the $d$-dimensional grid and similar graphs. Proc. Steklov Inst. Math., 2014, 285 (Suppl.~1), 169-182. http://mi.mathnet.ru/timm988

\bibitem{Watk71} Watkins, M.~E. On the action of non-abelian groups on graphs. J. Comb. Theory, 1971, 11, 95-104.

\bibitem{growth21} Zhong, X., Zhou, H., Li, C., Shtukenberg, A.~G., Ward, M.~D. and Kahr, B. Eshelby untwisting. Chem. Comm., 2021, 57, 5538-5541.
\end{small}

\end{thebibliography}
\end{document}